\documentclass[10pt, a4paper]{article}

\usepackage[tbtags]{amsmath}
\usepackage{amsthm, amssymb}
\usepackage{amscd, paralist}
\usepackage{graphicx}
\usepackage{bm}
\usepackage[title]{appendix}
\usepackage{comment}
\usepackage[colorlinks=true,allcolors=blue]{hyperref}
\usepackage{color}
\usepackage{enumitem}
\usepackage{url}

\numberwithin{equation}{section}

\theoremstyle{definition}

\newtheorem{ass}{Assumption}[section]

\newtheorem{thm}{Theorem}[section]
\newtheorem{prop}{Proposition}[section]
\newtheorem{lem}{Lemma}[section]
\newtheorem{rem}{Remark}[section]
\newtheorem{cor}{Corollary}[section]
\newtheorem{exam}{Example}[section]

\providecommand{\bra}[1]{\left( #1 \right)}
\providecommand{\seq}[1]{\left\{ #1 \right\}}
\providecommand{\abs}[1]{\left\lvert #1 \right\rvert}
\providecommand{\norm}[1]{\left\lVert #1 \right\rVert}

\renewcommand{\P}[1]{\mathbb{P}\left[ #1 \right]}
\providecommand{\E}[1]{\mathbb{E}\left[ #1 \right]}

\newcommand{\1}{\mbox{1}\hspace{-0.25em}\mbox{l}}

\newcommand{\supp}{\operatorname{supp}}

\newcommand{\Var}{\operatorname{Var}}

\newcommand{\R}{\mathbb{R}}

\newcommand{\al}{\alpha}
\newcommand{\be}{\beta}
\newcommand{\de}{\delta}
\newcommand{\De}{\Delta}
\newcommand{\ep}{\varepsilon}

\newcommand{\mA}{\mathcal{A}}
\newcommand{\mC}{\mathcal{C}}
\newcommand{\mG}{\mathcal{G}}
\newcommand{\mF}{\mathcal{F}}
\newcommand{\mN}{\mathcal{N}}

\providecommand{\eR}[1]{\mathcal{R}_n (#1, f_0)}
\providecommand{\heR}[1]{\mathcal{\widehat{R}}_n (#1, f_0)}
\providecommand{\weR}[1]{\mathcal{\widetilde{R}}_{n,s} (#1, f_0)}
\providecommand{\eQ}[1]{\mathcal{Q}_n (#1)}

\providecommand{\Psin}[1]{\Psi_n (#1)}
\providecommand{\PsiFn}[1]{\Psi_n^\mathcal{F} (#1)}
\newcommand{\Psibfn}{\bar{\Psi}_n (\hat{f}_n, \bar{f})}

\providecommand{\bemix}[1]{\beta_X (#1)}

\newcommand{\balpha}{\bm{\alpha}}
\newcommand{\bbeta}{\bm{\beta}}

\newcommand{\nmean}{\frac{1}{n} \sum_{k=0}^{n-1}}

\newcommand{\bd}{\mathbf{d}}
\newcommand{\bp}{\mathbf{p}}
\newcommand{\bt}{\mathbf{t}}
\newcommand{\bv}{\mathbf{v}}

\newcommand{\hf}{\hat{f}_n}
\newcommand{\tf}{\tilde{f}_n}
\renewcommand{\bf}{\bar{f}}

\newcommand{\hA}{\hat{A}}
\newcommand{\tA}{\tilde{A}}

\newcommand{\hM}{\hat{M}}
\newcommand{\tM}{\tilde{M}}
\newcommand{\tZ}{\tilde{Z}}

\newcommand{\X}{X_{k \De}}
\newcommand{\Y}{Y_{k \De}}
\newcommand{\I}{I_{k \De}}
\newcommand{\Z}{Z_{k \De}}

\newcommand{\ol}[1]{\overline{#1}}

\DeclareMathOperator{\Ep}{\mathbb{E}}

\def\besn#1{\begin{equation}\begin{split}#1\end{split}\end{equation}}

\allowdisplaybreaks

\title{Drift estimation for a multi-dimensional diffusion process using deep neural networks}
\author{Akihiro Oga\thanks{Graduate School of Mathematical Sciences, The University of Tokyo, 3-8-1 Komaba, Meguro-ku, Tokyo 153-8914 Japan}
\and
Yuta Koike\footnotemark[1]
}
\date{\today}

\begin{document}
\maketitle

\begin{abstract}
Recently, many studies have shed light on the high adaptivity of deep neural network methods in nonparametric regression models, and their superior performance has been established for various function classes. Motivated by this development, we study a deep neural network method to estimate the drift coefficient of a multi-dimensional diffusion process from discrete observations. We derive generalization error bounds for least squares estimates based on deep neural networks and show that they achieve the minimax rate of convergence up to a logarithmic factor when the drift function has a compositional structure. 
\smallskip

\end{abstract}

\section{Introduction}

Let us consider a $d$-dimensional diffusion process $X=\bra{X_t}_{t \geq 0}$ satisfying the following stochastic differential equation (SDE):
\begin{align}
  dX_t &= b (X_t) dt + \Sigma (X_t) dw_t,\qquad
  X_0=\eta,
  \label{eq.mod} 
\end{align}
where $b(x) = \bra{b^{i}(x)}_{1\leq i \leq d}$ is a $d$-dimensional vector,
$\Sigma(x) = \bra{\Sigma_{ij}(x)}_{1\leq i,j \leq d}$ a $d\times d$ matrix,
$\eta$ a $d$-dimensional random vector and $w_t$ a $d$-dimensional standard Wiener process independent of $\eta$.
Our aim is to pursue nonparametric estimation for the drift function $b$, given discrete observations of $X$.


In the univariate case, nonparametric drift estimation has been studied by many researchers.
See Kutoyants \cite{ref.1}, Hoffmann \cite{ref.2}, Spokoiny \cite{Sp00}, Gobet et al. \cite{GHR04}, Dalalyan \cite{Da05}, Comte et al. \cite{ref.3}, Comte and Genon-Catalot \cite{CoGC21}, among others.
In contrast, nonparametric estimation for multi-dimensional diffusions, especially from discrete observations, was rarely studied.
There are relatively many studies in a setting with a continuous record of the process available; see Dalalyan and Rei\ss~\cite{ref.4}, Strauch \cite{St15,St16}, and Nickl and Ray \cite{NiRa20}. 
When only discrete observations are available, Schmisser \cite{ref.ES} proposed penalized least squares estimators based on B-splines and established oracle bounds for the expected empirical errors.
Another relevant work is Bandi and Moloche \cite{BaMo18} who develop point-wise consistency and asymptotic normality of Nadaraya-Watson type estimators. 
In this paper, we employ deep neural networks instead of these methods.

In the past few years, many researchers have contributed to understand theoretical advantages of deep neural network estimates for nonparametric regression models. 
Schmidt-Hieber \cite{ref.SH} considered a situation where the regression function is a composition of several functions and showed that the minimax estimation rate is not attained by wavelet series estimators but attainable (up to a logarithmic factor) by least squares estimators based on deep neural networks with the ReLU activation function. 
See also \cite{BaKo19} for a related result. 
Suzuki \cite{ref.Suzuki} showed that such estimators also achieve the minimax optimal rate over a class of functions in a Besov space and its variant with mixed smoothness for which (nonadaptive) linear estimators are suboptimal; see also \cite{SuNi19,TsSu21}. 
Imaizumi and Fukumizu \cite{ref.7} considered estimation for piece-wise smooth functions,
and they showed that estimators using deep neural networks are minimax optimal and outperform linear estimators. 
This type of result was further explored by \cite{HaSu20,ImFu20}. 
In addition, several authors showed that deep neural network methods adapt to an intrinsic dimension of covariates; see \cite{ClKl21,NaIm20,SH19} for details. 

In spite of recent progress in deep neural network methods, they are hardly applied to statistical estimation for diffusion processes, and almost nothing is known about their theoretical properties in this context. One exception is work by Ogihara \cite{Og19} who considers neural network estimates for the diffusion matrices and studies their estimation errors as misspecified models. 
Unlike this work, we focus on drift estimation and develop a theory in a nonparametric way. 
As already mentioned above, the literature on nonparametric drift estimation based on discrete observations is rather scarce, and this article also sheds new light on this area.   


In this paper, we assume that the process is discretely observed: the discrete observations $(X_0, X_{\De}, \dots, X_{n\De})$ have a sampling interval $\De>0$. 
We are mainly interested in high-frequency and long-term sampling schemes such that $\De = \De_n$ tends to 0 and $n\De$ tends to infinity, but most results are stated in finite samples. 
The statistical problem is to recover the unknown drift function $b$ over a compact set $\mathcal{K}$ from the sample $\bra{\X}_{k=0}^{n}$.
By scaling the model appropriately, it suffices to consider the case $\mathcal{K} = [0,1]^d$ without loss of generality. Also, it is enough to estimate $b$ for each component because we can combine component-wise estimators to construct an estimator for $b$ itself. 
Specifically, we consider the $\mathsf i$-th component of $b$ for a given $\mathsf i \in \seq{1, \dots, d}$ and restrict the domain of the component function to the compact set $[0, 1]^d$,
that is $f_0 := b^{\mathsf i} \1_{[0,1]^d} $, and our aim is to estimate this unknown function $f_0$. 
This is standard in the literature; see e.g.~\cite{ref.3,ref.2,ref.ES}.
We establish an oracle inequality for estimators belonging to a class of deep neural networks analogous to the one established in \cite{ref.SH} for nonparametric regression models. 
Then, under a composition assumption on $f_0$ discussed in \cite{ref.SH}, we derive generalization error bounds for least squares estimates based on deep neural networks and show that they attain the minimax rate of convergence up to a logarithmic factor. 
For the latter purpose, we develop a technique to control $\beta$-mixing coefficients uniformly over a certain class of diffusion processes, which will be of independent interest; see Lemma \ref{lemma:beta} and its proof. 


This paper is organized as follows.
In Section \ref{sec.Definitions_and_Assumptions}, we collect some assumptions and definitions used throughout the paper. 
Section \ref{sec.main} presents the main results of this paper. 
All mathematical proofs are deferred to Sections \ref{sec.proof_oracle}--\ref{sec:proof-minimax}. 

\paragraph{Notation.}
For a vector or matrix $W$, we write $|W|$ for the Frobenius norm (i.e.~the Euclidean norm for a vector), $|W|_\infty$ for the maximum-entry norm and $|W|_0$ for the number of non-zero entries. 
For vectors $x, y\in\mathbb R^d$, $\langle x, y \rangle:=x^\top y$ denotes the inner product.
We write $\|f\|_\infty := \sup_{x \in [0,1]^d} |f(x)|$ for the supremum on the compact set $[0,1]^d$,
$\lfloor x \rfloor$ for the largest integer $\leq x$,
and $\lceil x \rceil$ for the smallest integer $\geq x$. 
For a random vector $\xi$, we denote by $\mathcal L_\xi$ the law of $\xi$. 


\section{Assumptions and definitions}
\label{sec.Definitions_and_Assumptions}

\subsection{Assumptions on the diffusion process model}
\label{sec.Assumptions_Process}

We will impose the following assumptions. 
\begin{ass}
  \label{ass.SDE1}
  The functions $b$ and $\Sigma$ are globally Lipschitz continuous: There are constants $C'_{b},C'_\Sigma>0$ such that 
  \begin{align*}
    & \abs{b(x)-b(y)} \leq C'_{b} |x-y|, \\
    & \abs{\Sigma(x)-\Sigma(y)} \leq C'_\Sigma |x-y|
  \end{align*}
  for all $x, y \in \R^d$. 
  In particular, $C_{b} := \sup_{x \in [0,1]^d} \abs{b(x)}<\infty$ and $C_\Sigma := \sup_{x \in [0,1]^d} \abs{\Sigma(x)}<\infty$.
\end{ass}

\begin{ass}
  \label{ass.mixing}
  The process $X$ is exponentially $\be$-mixing:
  there exist positive constants $C_\be, C'_\be$ such that, for all $t>0$,
  \begin{align*}
    \bemix{t} \leq C'_\be \exp \bra{- C_\be \, t}
  \end{align*}
  where $\bemix{t}$ is the $\be$-mixing coefficient of $X$ (see \cite{VoRo59} for the precise definition). 
\end{ass}


\begin{ass}
  \label{ass.SDE6}
  $\De \leq 1$ and $n\De \geq 2$.
\end{ass}


\begin{rem}\label{rem:mixing}
(a) Assumption \ref{ass.mixing} is implied by the following exponential ergodicity assumption on $X$: 
\begin{enumerate}[label=($\star$)]

\item \label{ass.ergodic} There is a probability measure $\Pi$ on $\mathbb R^d$, a function $V\in L^1(\Pi)\cap L^1(\mathcal L_\eta)$ and a constant $\rho>0$ such that
  \begin{equation*}
  \|P^t(x,\cdot)-\Pi\|\leq V(x)e^{-\rho t}
  \end{equation*}
  for all $x\in\mathbb R^d$ and $t\geq0$, where $(P^t)_{t\geq0}$ is the transition semigroup of $X$ and $\|\cdot\|$ is the total variation norm. 

\end{enumerate}
This follows from the same argument as in the proof of \cite[Proposition 3]{ref.Liebscher2005} (see Appendix \ref{proof:rem:mixing}). 
\smallskip

\noindent(b) Assumption \ref{ass.ergodic} is satisfied when $X$ satisfies the following conditions in addition to Assumption \ref{ass.SDE1}: 
  \begin{enumerate}[label=(\roman*)]
  
  \item There are constants $r>0$, $\alpha\geq1$ and $M_0>0$ such that $\langle b(x),x\rangle\leq -r|x|^{\alpha}$ for all $x\in\mathbb R^d$ with $|x|>M_0$. 
  
  \item There are constants $\lambda_{-}, \lambda_{+}>0$ such that $\lambda_{-}|x|^{2} \leq |\Sigma(x)^\top x|^2 \leq \lambda_{+}|x|^{2}$ for all $x \in \R^d$. 
  
  \item $\E{\exp (\nu \abs{\eta})} < \infty$ for some $\nu>0$. 
  
  \end{enumerate}
  This follows from Theorem 3.3.4 in \cite{Ku18}.
\end{rem}


\subsection{Multilayer neural networks}
\label{sec.multilayer_NN}

To fit a multilayer neural network, we need to choose an activation function $\sigma:\mathbb{R}\rightarrow \mathbb{R}$ and the network architecture.
Focusing on the importance in deep learning, we use the rectified linear unit (ReLU) activation function $$\sigma(x) = \max\{x,0\}.$$
For $\bv=(v_1, \ldots, v_r)\in \mathbb{R}^r,$ define the shifted activation function $\sigma_{\bv}: \mathbb{R}^r \rightarrow \mathbb{R}^r$ as 
\begin{align*}
  \sigma_{\bv}
  \left(
  \begin{array}{c}
  y_1 \\
  \vdots \\
  y_r
  \end{array}
  \right)
  =
  \left(
  \begin{array}{c}
  \sigma(y_1-v_1) \\
  \vdots \\
  \sigma(y_r-v_r)
  \end{array}
  \right).
\end{align*}
The network architecture $(L, \bp)$ consists of a positive integer $L$ called the {\it number of hidden layers} or {\it depth} and a {\it width vector} $\bp=(p_0, \ldots, p_{L+1}) \in \mathbb{N}^{L+2}.$ A neural network with network architecture $(L, \bp)$ is a function $f:\mathbb R^{p_0}\to\mathbb R^{p_{L+1}}$ of the form 
\begin{align}
  f(x) =
  W_L  \sigma_{\bv_L}   W_{L-1}  \sigma_{\bv_{L-1}}  \cdots  W_1 \sigma_{\bv_1}  W_0x,\qquad x \in \mathbb R^d,
  \label{eq.NN}
\end{align}
where $W_i$ is a $p_i \times p_{i+1}$ weight matrix and $\bv_i \in \mathbb{R}^{p_i}$ is a shift vector.
For $s,F\in(0,\infty]$, we set
%
\begin{multline}\label{eq.defi_bd_sparse_para_space}
  \mF(L, \bp, s, F)
  := \left\{f\1_{[0,1]^d}: \text{$f$ is of the form \eqref{eq.NN}},
\max_{j = 0, \ldots, L} |W_j|_{\infty} \vee |\bv_j|_\infty \leq 1,\right.
  \\
 \left. \sum_{j=0}^L \bra{|W_j|_0 + |\bv_j|_0} \leq s, \, \big\| |f|_\infty \big\|_\infty \leq F \right\},
\end{multline}
where we use the convention that $\bv_0$ is a vector with coefficients all equal to zero. 
We write $\mF(L, \bp, s) = \mF(L, \bp, s, \infty)$ for short. 
Since we are interested in estimating a real-valued function on $\mathbb R^d$, we always assume that $p_0=d$ and $p_{L+1}=1$ in the following. 
\begin{rem}\label{rem:dnn-b}
In this paper, following \cite{ref.SH}, we restrict our attention to neural networks with parameters bounded by 1. It is worth mentioning that any neural network with depth $L$ and parameters bounded by $B>1$ can be transformed into a neural network with depth $L':=(\lceil\log B\rceil+5)L$ and parameters bounded by 1; see Proposition A.3 in \cite{ElPeGrBo21}. Moreover, the proof of this proposition indicates that the transformed network has at most $2(s+12L')$ non-zero parameters, where $s$ is the number of non-zero parameters of the original network. 
\end{rem}

\begin{rem}[Activation function]
Although we focus only on the ReLU activation function for simplicity, our main results are still valid for other activation functions as long as they are Lipschitz continuous. 
To be precise, we need properties of the ReLU activation function only for the proofs of Lemma \ref{lem.entropy} (covering entropy bound for neural networks) and Lemma \ref{lem:approx} (approximation error bound by neural networks), and these results are still valid for Lipschitz continuous activation functions; see Proposition 2 and Theorem 9 in \cite{OhKi20} for details. 
\end{rem}

\subsection{Estimator}
\label{sec.Estimator}

In this paper, unless otherwise stated, an estimator refers to a real-valued random function $\hf$ defined on $\R^d$ such that the map $(\omega,x)\mapsto\hf(\omega,x)$ is measurable with respect to the product of the $\sigma$-algebra generated by $(X_{k\De})_{k=0}^n$ and the Borel $\sigma$-algebra of $\R^d$. 

To evaluate the statistical performance of an estimator $\hf$, we derive bounds for the generalization error defined by
\begin{align*}
  \eR{\hf} & := \E{\nmean \bra{\hf(\X')-f_0(\X')}^2 },
\end{align*}
where $X'=(X'_t)_{t\geq 0}$ is an independent copy of $X$.
\begin{rem}
If $X$ is stationary with invariant distribution $\Pi$, we have
\[
\eR{\hf}=\E{\int_{\mathbb R^d}\bra{\hf(x)-f_0(x)}^2\Pi(dx)}.
\]
That is, $\eR{\hf}$ is equal to the $L^2$-risk with respect to $\Pi$. 
If $X$ is not stationary but satisfies \ref{ass.ergodic}, one can prove (cf.~the proof of Lemma \ref{lem:reduce-to-l2})
\[
\left|\eR{\hf}-\E{\int_{\mathbb R^d}\bra{\hf(x)-f_0(x)}^2\Pi(dx)}\right|\leq\frac{4F^2\E{V(\eta)}}{n(1-e^{-\rho\De})},
\] 
provided that $|\hf|\leq F$ and $|f_0|\leq F$ for some $F>0$. 
In this case, we can also relate $\eR{\hf}$ to the so-called path-dependent generalization error:  
Following Eq.(4) of \cite{KuMo17}, we may define the path-dependent generalization error in our setting as
\[
\weR{\hf}:=\E{\bra{\hf(X_{n\De+s})-f_0(X_{n\De+s})}^2\mid \bra{\X}_{k=0}^{n}},
\]
where $s>0$ is fixed. This can be seen as an $s$-ahead prediction error of $\hf$ given actual observations of $\bra{\X}_{k=0}^{n}$. 
Thanks to the Markov property of $X$, $\weR{\hf}$ can be rewritten as
\[
\weR{\hf}=\int_{\mathbb R^d}\bra{\hf(x)-f_0(x)}^2P^s(X_{n\De},dx).
\]
Therefore, by \ref{ass.ergodic} and the boundedness assumption on $\hf$ and $f_0$, 
\[
\left|\weR{\hf}-\int_{\mathbb R^d}\bra{\hf(x)-f_0(x)}^2\Pi(dx)\right|
\leq4F^2V(X_{n\De})e^{-\rho s}.
\]
Thus we obtain
\[
\weR{\hf}\leq\eR{\hf}+4F^2V(X_{n\De})e^{-\rho s}+\frac{4F^2\E{V(\eta)}}{n(1-e^{-\rho\De})}.
\]
Note that $s$ should be large to make the second term on the right hand side small. This is analogous to bounds in \cite[Section 6]{KuMo17} since they also have the $\phi$-mixing coefficient $\phi(s)$ in their bounds and we usually need $s$ to be large to make $\phi(s)$ small. 
\end{rem}

%
%

Let $X^{\mathsf i}$ denote the $\mathsf i$-th component of $X$. 
Set
\[
\Y := \frac{1}{\De} \bra{X^{\mathsf i}_{(k+1)\De} - \X^{\mathsf i}},\qquad k=0,1,\dots,n-1.
\]
When $\De$ is small, $\Y$ is approximately equal to $b^{\mathsf i}(\X)$ plus a noise term (cf.~\eqref{eq:Y-reg}). This motivates us to consider a least squares estimator obtained by minimizing the following quantity subject to $f\in\mF(L, \bp, s, F)$:
\begin{equation*}
\eQ{f}:=\nmean \bra{\Y-f(\X)}^2.
\end{equation*}
In practice, however, it might be difficult to find an exact minimizer of this quantity. 
For this reason, we handle a general estimator $\hf$ and explicitly incorporate the optimization error into our error bounds. 
Given a pointwise measurable class $\mF$ consisting of real-valued functions on $\R^d$ (cf.~Example 2.3.4 in \cite{ref.Bernstein_ineq}), we define 
\begin{align}
  \PsiFn{\hf} &:= \E{ \eQ{\hf} - \inf_{\tilde{f} \in \mF} \eQ{\tilde f} }.
  \label{eq.estimator}
\end{align}
$\PsiFn{\hf}$ measures a gap between $\hf$ and an exact minimizer of $\eQ{f}$ subject to $f\in\mF$. We write $\Psin{\hf} = \Psi_n^{\mF(L,\bp,s,F)} (\hf)$ for short. 
Convergence analysis of the optimization error $\Psin{\hf}$ is a very active research area in the machine learning community. See e.g.~\cite{Sun20} for an overview. 


\begin{rem}[Choice of $\De$]
In real data application, we have the freedom to choose the time unit, and this choice affects the objective function $\mathcal Q_n(f)$ through the value of $\De$, so one may wonder whether this choice has an impact on the performance of the estimator (see \cite{EgMa19} for a related discussion). 
Indeed, in our case, the choice of the time unit affects only on the scale of the estimator, so the performance will not be sensitive to the choice as long as $F$ is taken sufficiently large. 
To see this, let us consider the case that we change the time scale from $\Delta$ to $\tau\Delta$ for some $\tau>0$. 
Note that by this change we should regard the underlying process as $X^\tau_t:=X_{t/\tau}$ instead of $X_t$ because the observation data $\bra{\X}_{k=0}^{n}$ are unchanged regardless of the choice of the time scale. It is straightforward to check that the new process $X^\tau=(X^\tau_t)_{t\geq0}$ satisfies the following SDE:
\[
dX^\tau_t = \tau^{-1}b (X^\tau_t) dt + \tau^{-1/2}\Sigma (X^\tau_t) dw^\tau_t,
\]
where $w^\tau_t:=\sqrt{\tau}w_{t/\tau}$ is a $d$-dimensional standard Wiener process. Therefore, by this time scale change, the drift function to estimate is also changed to $\tau^{-1}b$. 
With this in mind, we see how the time scale change affects minimizers of objective functions. It alters $\mathcal Q_n(f)$ to
\[
\mathcal Q_n^\tau(f):=\nmean \bra{\tau^{-1}\Y-f(\X)}^2=\tau^{-2}\mathcal Q_n(\tau f).
\]
If $\tau\leq1$, we have $\{\tau f:f\in\mF(L, \bp, s, F)\}\subset\mF(L, \bp, s, F)$. 
Therefore, if a minimizer $\hf$ of $\mathcal Q_n(f)$ over $f\in\mF(L, \bp, s, F)$ is realized in $\{\tau f:f\in\mF(L, \bp, s, F)\}$, then $\tau^{-1}\hf$ gives a minimizer of $\mathcal Q_n^\tau(f)$ over $f\in\mF(L, \bp, s, F)$, so the time scale change does not matter. The former likely happens if we take $F$ sufficiently large. If $\tau>1$, we reach the same conclusion by exchanging the roles of $\mathcal Q_n(f)$ and $\mathcal Q_n^\tau(f)$. 
\end{rem}

\begin{rem}[Extension to L\'evy-driven SDEs]
Our estimation method is formally applicable when the driving process $w_t$ is a more general L\'evy process, and we conjecture that this would be theoretically justified as long as increments of the driving L\'evy process have sufficiently light tails. 
However, the proof will need the following non-trivial modification: In our proof, the fact that the driving process is a \textit{continuous} local martingale is essentially used for applications of Lemma \ref{lemma:sn}, a self-normalized inequality. It will be still possible to use Proposition 4.2.1 in \cite{BJY86} instead of Lemma 1.2 in \cite{ref.VictorH}, which extends the scope of Lemma \ref{lemma:sn} to locally square-integrable martingales but replaces $\langle M\rangle_T$ by $\langle M^c\rangle_T+\sum_{0\leq t\leq T}(\Delta M_t)^2$, where $M^c$ is the continuous martingale part of $M$. So we additionally need to control the sum of squared jumps term $\sum_{0\leq t\leq T}(\Delta M_t)^2$. In our original proof, it is crucial that $\langle M\rangle_T$ enjoys a good pathwise bound, but this is generally false for $\sum_{0\leq t\leq T}(\Delta M_t)^2$. We conjecture that such a pathwise bound would be available after truncation, and the truncation error would be negligible as long as increments of the driving L\'evy process have sufficiently light tails. 
In contrast, if the increments have heavy tails, the truncation error would be problematic, so we speculate that it would be better to employ a different estimation method such as the least truncated squares in \cite{Sc19} or the least absolute deviations in \cite{Ma10}. 
Since one major motivation for using non-Gaussian processes is to model heavy tailed data, we leave this subject to future work. 
\end{rem}


\section{Main results}
\label{sec.main}

Our first main result is the following oracle type inequality. 
\begin{thm}
  \label{thm.oracle_ineq}
  Suppose that Assumptions \ref{ass.SDE1}--\ref{ass.SDE6} are satisfied. 
  Consider the $d$-dimensional diffusion process model \eqref{eq.mod} satisfying $\|f_0 \|_\infty \leq F$ for some $F\geq 1$.
  Let $\hf$ be any estimator taking values in the class $\mF(L, \bp, s, F)$ satisfying $s\geq 2$.
  Then for any $\tau>1$, there exists a constant $C_\tau$ depending only on $(C_{b}, C_{b}', C_\Sigma, C_\Sigma', C_\be, C'_\be, \tau)$ such that
  \begin{align}
    \eR{\hf}
    \leq \tau & \bra{ \Psin{\hf} + \inf_{f\in \mF(L, \bp, s, F)} \norm{f - f_0}_\infty^2 } \notag\\
    & + C_{\tau} F^2 \bra{ \frac{s (L\log s + \log (n\De)) \cdot \log (n\De)}{n\De} + \De }.\label{eq:oracle}
  \end{align}
\end{thm}

\begin{rem}
(a) Theorem \ref{thm.oracle_ineq} can be seen as an analog of \cite[Theorem 2]{ref.SH} in the drift estimation problem for a discretely observed diffusion process. In contrast to the latter result, our bound has an extra logarithmic factor $\log(n\De)$ in the third term. This is caused by the fact that $\X$ are not independent. 
Also, the term $\De$ appears in the bound due to discretization errors. 

\noindent(b) Our bound also shares a similar nature with the bound given in \cite[Theorem 1]{ref.ES}. An important difference is that the latter is concerned with the expected empirical error. Indeed, our proof allows us to establish a bound for the expected empirical error without the $\beta$-mixing property; see Lemma \ref{lem.oracle_ineq2}. 

\noindent(c) Note that the boundedness assumption $\|f_0\|_\infty\leq F$ is imposed on $f_0 = b^{\mathsf i} \1_{[0,1]^d} $ rather than $b$ itself. Since we assume $b$ is (Lipschitz) continuous, this condition always holds with \textit{some} $F$.  
\end{rem}

Recent progress in neural network theory allows us to get a sharp bound for the approximation error 
\[
\inf_{f\in \mF(L, \bp, s, F)} \norm{f - f_0}_\infty
\]
under various assumptions on $f_0$; see e.g.~references cited in the Introduction. 
As an illustration, we apply Theorem \ref{thm.oracle_ineq} to the situation where the drift function has a compositional structure as discussed in \cite{ref.SH}; see Appendix \ref{sec:besov} for the case that the drift function belongs to a Besov space. 

For $p, r\in\mathbb N$ with $p\geq r$, $\beta,K>0$ and $l<u$, we denote by $\mC_r^\beta([l,u]^p, K)$ the set of functions $f:[l,u]^p\to\mathbb R$ satisfying the following conditions:
\begin{enumerate}[label=(\roman*)]

\item $f$ depends on at most $r$ coordinates.

\item $f$ is of class $C^{\lfloor\beta\rfloor}$ and satisfies
\[
\sum_{\balpha : |\balpha|_1 < \beta}\|\partial^{\balpha} f\|_\infty + \sum_{\balpha : |\balpha |_1= \lfloor \beta \rfloor } \, \sup_{\stackrel{x, y \in [l,u]^p}{x \neq y}}
  \frac{|\partial^{\balpha} f(x) - \partial^{\balpha} f(y)|}{|x-y|_\infty^{\beta-\lfloor \beta \rfloor}} \leq K,
\]
where we used multi-index notation, that is, $\partial^{\balpha}= \partial^{\alpha_1}\cdots \partial^{\alpha_p}$ with $\balpha = (\alpha_1, \ldots, \alpha_p) \in \mathbb{Z}_{\geq 0}^p$ and $|\balpha|_1:=\sum_{j=1}^p\alpha_j.$

\end{enumerate}
Let $\bd=(d_0,\ldots,d_{q+1})\in\mathbb N^{q+2}$ with $d_0=d$ and $d_{q+1}=1$, $\bt=(t_0, \ldots, t_q)\in\mathbb N^{q+1}$ with $t_i\leq d_i$ for all $i$ and $\bbeta=(\beta_0, \ldots, \beta_q)\in(0,\infty)^{q+1}$. 
We define $\mG\big(q, \bd, \bt, \bbeta, K\big)$ as the class of functions $f:[0,1]^d\to\mathbb R$ of the form
\begin{equation}\label{eq.mult_composite_regression}
f= g_q \circ \cdots \circ g_0,
\end{equation}
where $g_i=(g_{ij})_j : [l_i, u_i]^{d_i}\rightarrow [l_{i+1},u_{i+1}]^{d_{i+1}}$ with $g_{ij} \in \mC_{t_i}^{\beta_i}\big([l_i,u_i]^{d_i}, K\big)$ for some $|l_{i+1}|, |u_{i+1}| \leq K$, $i=0,\dots,q$; in particular, $g_{ij}$ depends on at most $t_i$ coordinates. 
The function class $\mG\big(q, \bd, \bt, \bbeta, K\big)$ is rich enough to cover many important models such as (generalized) additive models and functions with sparse tensor decomposition; see Section 4 in \cite{ref.SH} for more details. 

Denote by $\ol\mG\big(q, \bd, \bt, \bbeta, K\big)$ the class of all functions $f:\mathbb R^d\to\mathbb R$ such that $\supp(f)\subset[0,1]^d$ and the restriction of $f$ to $[0,1]^d$ belongs to $\mG\big(q, \bd, \bt, \bbeta, K\big)$. 
Also, define
\begin{align*}
  \beta_i^* &:= \beta_i \prod_{\ell=i+1}^q (\beta_{\ell}\wedge 1),\qquad
  \phi_n := \max_{i=0, \ldots, q } (n\De)^{-\frac{2\beta_i^*}{2\beta_i^*+ t_i}},
\end{align*}
and
\[
c_L^l := \sum_{i=0}^q \frac{\beta_i + t_i}{t_i} \log_2 (4t_i\vee 4\beta_i).
\]

\begin{thm}
  \label{thm.main}
  Suppose that Assumptions \ref{ass.SDE1}--\ref{ass.SDE6} are satisfied. 
  Consider the $d$-dimensional diffusion process model \eqref{eq.mod} such that $f_0\in\ol\mG(q, \bd, \bt, \bbeta, K)$. 
  Let $\hf$ be an estimator taking values in the network class $\mF(L, \bp, s, F)$ satisfying
  \begin{enumerate}[label=(\roman*)]
    \item $F \geq K\vee1,$
    \item $11q+8+ c_L^l \log_2 (n\De\phi_n) \leq L \leq c_L^u n\De \phi_n,$
    \item $c_{\bp} n\De \phi_n \leq \min_{i=1, \ldots, L} p_i,$
    \item $c_s^l n\De \phi_n \log (n\De) \leq s \leq c_s^u n\De \phi_n \log (n\De),$
  \end{enumerate}
  for some constants $c_L^u, c_{\bp}, c_s^l, c_s^u>0$.
  Then there exist constants $C, C', C''$ depending only on $(C_{b}, C_{b}', C_\Sigma, C_\Sigma', C_\be, C'_\be, c_L^u, c_{\bp}, c_s^l, c_s^u, q, \bd, \bt, \bbeta, F)$ such that
  \begin{align}
    \eR{\hf}
    \leq C \Psin{\hf} + C' \phi_n L \log^3 (n\De) + C'' \De.
    \label{eq.main1}
  \end{align}
\end{thm}
The error bound in \eqref{eq.main1} consists of three terms. 
The first term $\Psin{\hf}$ is equal to 0 when $\hf$ minimizes $\eQ{f}$ over $f\in\mF(L, \bp, s, F)$. 
The second term $\phi_nL\log^3(n\De)$ is minimized when $L$ is taken of the order of $\log_2 (n\De)$. 
Finally, in the standard asymptotic regime such that $n\De\to\infty$ and $n\De^2\to0$, the third term $\De$ is dominated by $\phi_n$. This leads to the following corollary.  
\begin{cor}\label{coro.main}
Under the assumptions of Theorem \ref{thm.main}, let $\hf$ be a minimizer of $\eQ{f}$ over $f\in\mF(L, \bp, s, F)$. Suppose that $L=O(\log_2(n\De))$, $n\De\to\infty$ and $n\De^2\to0$ as $n\to\infty$. Then
\[
    \eR{\hf} = O\bra{\phi_n \log^4 (n\De)}\qquad\text{as }n\to\infty.
\]
\end{cor}

\begin{exam}[Additive model]\label{ex:additive}
Suppose that the drift function $b$ is of the form $b(x_1,\dots,x_d)=\sum_{j=1}^db_j(x_j)$ for some Lipschitz functions $b_j:\mathbb R\to\mathbb R^d$, $j=1,\dots,d$. Then $b$ is also Lipschitz continuous. Therefore, if $\Sigma$ is Lipschitz continuous and the assumptions in Remark \ref{rem:mixing}(b) hold, Assumptions \ref{ass.SDE1} and \ref{ass.mixing} are satisfied. 
Let $\beta\geq1$ be a constant and assume that every $b_j$ is of class $C^{\lfloor \beta\rfloor}$ and its $\lfloor\beta\rfloor$-th derivative is $(\beta-\lfloor\beta\rfloor)$-H\"older continuous. 
In this case, we have $\phi_n=(n\De)^{-2\beta/(2\beta+1)}$ (cf.~Section 4 in \cite{ref.SH}), so we obtain $\eR{\hf} = O\bra{(n\De)^{-2\beta/(2\beta+1)} \log^4 (n\De)}$ as $n\to\infty$ in the setting of Corollary \ref{coro.main}. 
\cite{ref.SH} has obtained a similar convergence rate for the expected empirical error by B-spline regression while assuming the drift function has an additive form. Unlike this result, we do not need this information to construct our estimators. 
\end{exam}

We now show that the rate of convergence given by Corollary \ref{coro.main} is optimal up to the logarithmic factor in the minimax sense. For this purpose, we introduce a class of diffusion processes over which we can control the $\beta$-mixing coefficients uniformly while $\phi_n$ gives a lower bound on the minimax rate of convergence for $\eR{\hf}$. 
For $r\in(0,1]$ and $K>0$, we denote by $\mathcal B_0(d,r,K)$ the class of measurable functions $b:\mathbb R^d\to\mathbb R^d$ satisfying the following conditions:
\begin{align*}
\langle b(x),x\rangle\leq -r|x|\quad\text{for all }|x|>2d\quad\text{and}\quad
\sup_{x\in\mathbb R^d}|b(x)|_\infty\leq K.
\end{align*}
For each $b\in\mathcal{B}_0(d, r, K)$, we consider the following SDE:
\begin{equation}\label{eq:sde}
dX_t=b(X_t)dt+dw_t,\qquad X_0=\eta.
\end{equation}
This SDE has a unique weak solution by assumption (cf.~\cite[Theorems 7.2.1]{StVa97}). 
The following lemma shows that the diffusion model \eqref{eq:sde} satisfies Assumption \ref{ass.mixing} ``uniformly'' over $b\in\mathcal B_0(d,r,K)$. 
\begin{lem}\label{lemma:beta}
Consider the $d$-dimensional diffusion process model \eqref{eq:sde} with $b\in\mathcal B_0(d,r,K)$ and $\E{\exp(r|\eta|)}<\infty$. 
There are positive constants $C_\beta$ and $C_\beta'$ depending only on $d,r,K$ and $\eta$ such that
\begin{equation}\label{eq:minimax-beta}
\beta_X(t)\leq C'_\beta e^{-C_\beta t}\qquad\text{for all }t\geq0.
\end{equation}
\end{lem}
We define $\mathcal{B}^{\mathsf i}(r, q, \bd, \bt, \bbeta, K)$ as the class of all functions $b\in\mathcal B_0(d,r,K)$ satisfying the following conditions:
\begin{enumerate}[label=(\roman*)]

\item $|b(x)-b(y)|\leq K|x-y|$ for all $x,y\in\R^d$.

\item $f_0=b^{\mathsf i}\1_{[0,1]^d}\in\ol\mG\big(q, \bd, \bt, \bbeta, K\big)$.  

\end{enumerate}
Theorem \ref{thm.main} and Lemma \ref{lemma:beta} imply the following result.
\begin{prop}\label{prop:minimax}
Consider the $d$-dimensional diffusion process model \eqref{eq:sde} with $b\in\mathcal{B}^{\mathsf i}(r, q, \bd, \bt, \bbeta, K)$ and $\E{\exp(r|\eta|)}<\infty$. 
Let $\hf$ be a minimizer of $\eQ{f}$ over $f\in\mF(L, \bp, s, F)$ satisfying conditions (i)--(iv) in Theorem \ref{thm.main}. Then
\[
    \sup_{b\in\mathcal{B}^{\mathsf i}(r, q, \bd, \bt, \bbeta, K)}\eR{\hf} = O\bra{\phi_n \log^4 (n\De)}\qquad\text{as }n\to\infty,
\]
provided that $L=O(\log_2(n\De))$, $n\De\to\infty$ and $n\De^2\to0$.
\end{prop}



Let $i^*$ be an index such that $\beta^*_{i^*}/(2\beta_{i^*}^*+t_{i^*})=\min_{i=0,1,\dots,q}\beta^*_{i}/(2\beta_{i}^*+t_{i})$. Write $\beta^*:=\beta_{i^*}$ and $t^*:=t_{i^*}$. 
The next theorem shows that $\phi_n$ is a lower bound on the minimax rate of convergence for the generalization error over $\mathcal{B}^{\mathsf i}(r, q, \bd, \bt, \bbeta, K)$ for sufficiently large $K$ under natural constraints on $\bd,\bt$ and $\bbeta$. 
\begin{thm}\label{thm:minimax}
Consider the $d$-dimensional diffusion process model \eqref{eq:sde} with $b\in\mathcal{B}^{\mathsf i}(r, q, \bd, \bt, \bbeta, K)$ and $\E{\exp(r|\eta|)}<\infty$. 
Suppose that $n\De\to\infty$ and $\De\to0$ as $n\to\infty$. 
Suppose also that $\beta^*\geq1$ and $t_j\leq\min\{d_0,\dots,d_{j-1}\}$ for all $j$. Then, there is a constant $K>0$ such that 
\[
\liminf_{n\to\infty}\phi_n^{-1}\inf_{\hf}\sup_{b\in\mathcal{B}^{\mathsf i}(r, q, \bd, \bt, \bbeta, K)}\eR{\hf}>0,
\]
where the infimum is taken over all estimators $\hf$. 
\end{thm}
The assumption $\beta^*\geq1$ is necessary to ensure that the class $\mathcal{B}^{\mathsf i}(r, q, \bd, \bt, \bbeta, K)$ contains sufficiently many Lipschitz continuous functions. 
The assumption on $\bd$ and $\bt$ is the same one as in \cite[Theorem 3]{ref.SH}. 

Proposition \ref{prop:minimax} and Theorem \ref{thm:minimax} show that any minimizer of $\mathcal Q_n(f)$ over $f\in\mF(L, \bp, s, F)$ is minimax optimal up to the $\log^4(n\De)$ factor in the setting described in these results. So far, we have no idea how to remove the (presumably) extra logarithmic factor. 
%

\begin{rem}
Our estimator in Corollary \ref{coro.main} is not adaptive in the sense that the sparsity parameter $s$ should be selected of the order $O(n\De\phi_n\log(n\De))$ which depends on $\bt$ and $\bbeta$. 
We might be able to resolve this problem by introducing a penalty term when minimizing $\eQ{f}$ as suggested by \cite{OhKi20} for nonparametric regression models. We leave such an extension to future work. 
\end{rem}

\begin{rem}
Some authors have recently showed that deep neural networks have good ability to approximate \textit{discontinuous} functions; see \cite{HaSu20,ref.7,ImFu20} for example. 
However, we cannot benefit from these results because we encounter a fundamental difficulty specific to diffusion process models when the drift function is discontinuous: In our argument, the Lipschitz continuity of the drift function plays a crucial role to control discretization errors, entering the final bound as the last term on the right hand side of \eqref{eq:oracle}. 
To be precise, the Lipschitz continuity is essentially used in the proof of Lemma \ref{lemma:gronwall}. We conjecture that it could be relaxed to a H\"older continuity by replacing Gronwall’s inequality in this proof with Bihari's inequality, although this will change the final bound. 
In contrast, when the drift function is discontinuous, the current proof strategy will no longer work, so we will need a fundamental modification. We leave this topic to future research. 
\end{rem}


\section{Proof of Theorem \ref{thm.oracle_ineq}}
\label{sec.proof_oracle}

To prove Theorem \ref{thm.oracle_ineq}, we establish two general error bounds in nonparametric drift estimation for a discretely observed diffusion process, which may be of independent interest. 

Consider a pointwise measurable class $\mF$ of real-valued functions on $\mathbb R^d$. 
For $\delta>0$, a subset $\mG\subset\mF$ is called a \textit{$\de$-net} of $\mF$ with respect to $\|\cdot \|_\infty$ if for every $f\in\mF$ there is a function $g\in\mG$ such that $\|f-g\|_\infty\leq\de$. 
The smallest possible cardinality of a $\delta$-net of $\mF$ with respect to $\|\cdot \|_\infty$ is called the \textit{covering number} of $\mF$ with respect to $\|\cdot \|_\infty$ and denoted by $\mN(\delta, \mF, \|\cdot \|_\infty)$. 

Let $f_0:\mathbb R^d\to\mathbb R$ be a measurable function we wish to estimate. 
We consider an estimator $\hf$ taking values in $\mF$, and define its expected empirical error by
  \begin{align*}
    \heR{\hf} & := \E{\nmean \bra{\hf(\X)-f_0(\X)}^2 }.
  \end{align*}
  The first bound relates the generalization error $\eR{\hf}$ to its empirical counterpart $\heR{\hf}$. 
\begin{lem}
  \label{lem.oracle_ineq}
  Let $X=(X_t)_{t\geq0}$ be a general $d$-dimensional stochastic process that is not necessarily a diffusion process satisfying \eqref{eq.mod}. 
  Let $\delta>0$ and $\mN_n$ be an integer such that $\mN_n \geq \mN(\delta, \mF, \|\cdot\|_\infty)\vee\exp (7.4)$. Also, let $a_n$ be a positive number such that $\mu_n := \lfloor n/2a_n\rfloor>0$. 
  In addition, suppose that there is a number $F\geq1$ such that $|f|\leq F$ for all $f\in\mF\cup\{f_0\}$. Then, for all $\ep \in (0,1]$, 
  \begin{multline*}
    \eR{\hf}
    \leq (1+\ep) \heR{\hf}  + \frac{24(1+\ep)^2}{\ep} \cdot F^2 \frac{\log \mN_n}{\mu_n} \\
     + 4(2+\ep)F^2 \bemix{a_n \De} +\frac{8F^2}{n} + 4(2+\ep)F\de.
  \end{multline*}
\end{lem}


The second bound is an oracle type inequality for $\heR{\hf}$. 
\begin{lem}
  \label{lem.oracle_ineq2}
  Let $X$ be a $d$-dimensional diffusion process satisfying \eqref{eq.mod}, and set $f_0:=b^{\mathsf i}\1_{[0,1]^d}$. 
  Also, let $\delta>0$ and set $\mN_n:=\mN(\delta, \mF, \|\cdot\|_\infty)$. 
  Suppose that there is a number $F\geq1$ such that $|f|\leq F$ for all $f\in\mF\cup\{f_0\}$. 
  Suppose also that $\supp (f) \subset [0,1]^d$ for all $f \in \mF$. 
  Then, under Assumptions \ref{ass.SDE1} and \ref{ass.SDE6}, for all $\ep\in (0,1)$ there exists a constant $C_\ep$ depending only on $(C_{b}, C_{b}', C_\Sigma, C_\Sigma', \ep)$ such that
  \begin{align*}
    \heR{\hf}
    & \leq \frac{1}{1-\ep} \PsiFn{\hf} + \frac{1+\ep}{1-\ep} \inf_{f\in \mF} \eR{f} + C_\ep F^2 \gamma_{\de, n},
  \end{align*}
  where 
  \[
   \gamma_{\de, n}:= \De + \de + \frac{1}{\sqrt{n\De}} \int_0^\de \sqrt{\log \mN (u, \mF, \norm{\cdot}_\infty)} du + \frac{\log\mN_n+1}{n\De}.
  \]
\end{lem}

\begin{rem}
It is worth mentioning that Lemma \ref{lem.oracle_ineq2} does not require Assumption \ref{ass.mixing}. That is, $X$ is not necessarily $\beta$-mixing.  
\end{rem}

The remainder of this section proceeds as follows. We prove Lemmas \ref{lem.oracle_ineq} and \ref{lem.oracle_ineq2} in Sections \ref{sec:proof-mixing} and \ref{sec:proof-martingale}, respectively. Then, we give a proof of Theorem \ref{thm.oracle_ineq} in Section \ref{sec:proof-oracle_ineq}. 


\subsection{Proof of Lemma \ref{lem.oracle_ineq}}\label{sec:proof-mixing}


Since $\mN_n\geq\mN(\delta,\mF,\|\cdot\|_\infty)$, we can take a $\delta$-net $\{f_1,\dots,f_{\mN_n}\}$ of $\mF$ with respect to $\|\cdot\|_\infty$. 
By definition there is a random variable $j^*$ such that $|\hf-f_{j^*}| \leq \delta$.

We follow similar arguments as in Part (I) of the proof of \cite[Lemma 4]{ref.SH}, but we apply a blocking technique as in \cite{Yu94} to address the dependence of observations. 


\subsubsection{Step 1: $\be$-mixing}

For $l=0,1,\dots,\mu_n-1$, define
\begin{align*}
  H_l & := \seq{ k : 2l a_n \leq k \leq (2l+1) a_n - 1  }, \\
  H_{2,l} & := \seq{ k : (2l+1) a_n \leq k \leq 2(l+1) a_n - 1  }, \\
  \tilde{g}_l &:=
  \left(
  \begin{array}{c}
    \sum_{k \in H_{l}} ((f_1-f_0)(\X))^2 \\
    \sum_{k \in H_{l}} ((f_2-f_0)(\X))^2 \\
    \vdots \\
    \sum_{k \in H_{l}} ((f_{\mN_n}-f_0)(\X))^2 \\
  \end{array}
  \right), \\
  \tilde{g}'_l &:=
  \left(
  \begin{array}{c}
    \sum_{k \in H_{l}} ((f_1-f_0)(\X'))^2 \\
    \sum_{k \in H_{l}} ((f_2-f_0)(\X'))^2 \\
    \vdots \\
    \sum_{k \in H_{l}} ((f_{\mN_n}-f_0)(\X'))^2 \\
  \end{array}
  \right).
\end{align*}
In the following, extending the probability space if necessary, we assume that there is a sequence $(U_l)_{l=1}^\infty$ of i.i.d.~uniform variables over $[0,1]$ independent of $X$. 




\begin{prop}
  There exists a sequence $\bra{g_l}_{l=0}^{\mu_n -1}$ of independent random vectors in $\mathbb R^d$ such that each $g_l$ has the same law as $\tilde{g}_l$ and satisfies $\P{g_l \neq \tilde{g}_l} \leq \bemix{a_n \De}$. 
  Also, there exists a sequence $\bra{g'_l}_{l=0}^{\mu_n -1}$ of independent random vectors in $\mathbb R^d$ such that each $g'_l$ has the same law as $\tilde{g}'_l$ and satisfies $\P{g'_l \neq \tilde{g}'_l} \leq \bemix{a_n \De}$. 
\end{prop}

\begin{proof}
  For all $l_1, l_2$, define the $\sigma$-field $\mA_{(l_1, l_2)}$ generated by $\bra{\X}_{k \in H(l_1,l_2)}$, where $H(l_1,l_2) := \cup_{l=l_1}^{l_2} H_l$.
  For all $l$, we can find by definition that the $\sigma$-field generated by $\tilde{g}_l$ is a subset of $\mA_{(l,l)}$, i.e.~$\sigma \bra{\tilde{g}_l} \subset \mA_{(l,l)}$.
  Using Lemma 5.1 in \cite{ref.Rio}, there exists a random vector $g_l$, with the same law as $\tilde{g}_l$,
  independent of $\mA_{(0, l-1)}$, such that $\P{g_l \neq \tilde{g}_l} \leq \bemix{a_n \De}$.
  Furthermore $g_l$ is measurable with respect to the $\sigma$-field generated by $\mA_{(0, l)}$ and $U_l$.
  Therefore, for any $l$, $g_{l}$ is independent of $\bra{g_{l'}}_{l'=0}^{l-1}$,
  since for $l_1, l_2$ with $l_1 < l_2$, $g_{l_2}$ is independent of the $\sigma$-field generated by $\mA_{(0, l_1)}$ and $U_{l_1}$. 
  This completes the proof of the first claim. We can similarly prove the second claim. 
\end{proof}

Define
\begin{align*}
  g_{*, l} &:=g_{j^*,l},\qquad
  g'_{*, l} := g'_{j^*,l},\\
  \tilde{g}_{*, l} &:= \tilde{g}_{j^*, l}
  = \sum_{k \in H_{l}} ((f_{j^*}-f_0)(\X))^2, \\
  \tilde{g}'_{*, l} &:= \tilde{g}'_{j^*, l}
  = \sum_{k \in H_{l}} ((f_{j^*}-f_0)(\X'))^2,
\end{align*}
where $g_{j, l}$ and $g'_{j, l}$ denote the $j$-th components of $g_l$ and $g'_l$, respectively.

\begin{lem}
  \label{lem.5.4}
  For all $l=0,1,\dots,\mu_n-1$,
  \begin{align*}
    \abs{\E{g_{*, l}} - \E{\tilde{g}_{*, l}}}
    & \leq 4F^2 a_n \bemix{a_n \De}, \\
    \abs{\E{g'_{*, l}} - \E{\tilde{g}'_{*, l}}}
    & \leq 4F^2 a_n \bemix{a_n \De}.
  \end{align*}
\end{lem}

\begin{proof}
By definition we have $|\tilde g_{j,l}|\leq 4F^2a_n$ for all $j$. Since $g_l$ has the same law as $\tilde g_l$, we also have $|g_{j,l}|\leq4F^2a_n$ a.s.~for all $j$. Consequently, 
  \begin{align*}
    \abs{\E{g_{*, l}} - \E{\tilde{g}_{*, l}}}
    & \leq \E{\abs{ g_{*, l} - \tilde{g}_{*, l} } \1_{\seq{g_{*, l} \neq \tilde{g}_{*, l}}}} 
     \leq 4F^2 a_n \P{g_{*, l} \neq \tilde{g}_{*, l}} \\
    & \leq 4F^2 a_n \P{g_l \neq \tilde{g}_l} 
     \leq 4F^2 a_n \bemix{a_n \De}.
  \end{align*}
  Hence we obtain the first inequality. We can similarly prove the second one. 
\end{proof}


\subsubsection{Step 2: Bernstein's inequality}

For $j=1,\dots,\mathcal N_n$, define
\begin{align*}
  r_j & := \bra{\frac{4F^2 a_n \log \mN_n}{\mu_n} \vee \frac{1}{\mu_n} \sum_{l=0}^{\mu_n-1} \E{g'_{j, l}}}^{1/2}, \\
  B_j &:= \abs{\frac{\sum_{l=0}^{\mu_n-1} \bra{g'_{j, l} - g_{j, l}}}{2F r_j}}.
\end{align*}
Then we set
\begin{align*}
r^* & := r_{j^*} = \bra{\frac{4F^2 a_n \log \mN_n}{\mu_n} \vee \frac{1}{\mu_n} \sum_{l=0}^{\mu_n-1} \E{g'_{*, l} \left| \seq{X_t}_{t\geq 0} \right.}}^{1/2},\\
B &:= \max_{1\leq j\leq\mathcal N_n} B_j.
\end{align*}

\begin{lem}
  \label{lem.bernstein_ineq}
  If $\log \mN_n \geq 7.4$,
  \begin{align*}
    \E{B}
    & \leq 3 \sqrt{a_n \mu_n \log \mN_n},\qquad
    \E{B^2}
     \leq 12 a_n \mu_n \log \mN_n.
  \end{align*}
\end{lem}

\begin{proof}
  Define $\gamma := \frac{1 + \sqrt{37}}{3}$ and $\alpha=\alpha_n := \gamma \sqrt{a_n \mu_n \log \mN_n}$. 
  Note that $\gamma$ solves the equation $3\gamma^2-2\gamma-12=0$. 
  For all $j,l$, we have by construction
  \begin{align*}
    \abs{\frac{g'_{j, l} - g_{j, l}}{2F r_j}}
    & \leq \frac{\bra{2F}^2 a_n}{2F \sqrt{\frac{4F^2 a_n \log \mN_n}{\mu_n}}} = \sqrt{\frac{a_n \mu_n}{\log \mN_n}}
  \end{align*}
  and
  \begin{align*}
    \Var \bra{\frac{g'_{j, l} - g_{j, l}}{2F r_j}}
    & \leq \bra{2F r_j}^{-2} \cdot 2 \E{\bra{g'_{j, l}}^2} \\
    & \leq \frac{2 \cdot 4F^2 a_n \E{g'_{j, l}}}{4F^2 \frac{1}{\mu_n} \sum_{l'=0}^{\mu_n-1} \E{g'_{j, l'}}} 
     = \frac{2 a_n \mu_n \E{g'_{j, l}}}{\sum_{l'=0}^{\mu_n-1} \E{g'_{j, l'}}}.
  \end{align*}
  Hence
  \begin{align*}
    \sum_{l=0}^{\mu_n-1} \Var \bra{\frac{g'_{j, l} - g_{j, l}}{2F r_j}}
    & \leq 2 a_n \mu_n.
  \end{align*}
  Using Bernstein's inequality (cf. Lemma 2.2.9 in \cite{ref.Bernstein_ineq}), we have for all $ t \geq \alpha$
  \begin{align*}
    \P{B_j \geq t}
    & \leq 2 \exp \bra{-\frac{\frac{1}{2} t^2}{\frac{1}{3} \sqrt{\frac{a_n \mu_n}{\log \mN_n}} t + 2 a_n \mu_n}} \\
    & \leq 2 \exp \bra{-\frac{\frac{1}{2} \gamma \sqrt{a_n \mu_n \log \mN_n}}{\frac{1}{3} \sqrt{\frac{a_n \mu_n}{\log \mN_n}} \gamma \sqrt{a_n \mu_n \log \mN_n} + 2 a_n \mu_n} t} \\
    & = 2 \exp \bra{-\frac{3 \gamma^2}{2 \gamma + 12} \frac{\log \mN_n}{\al} t} 
     = 2 \exp \bra{-\frac{\log \mN_n}{\al} t}.
  \end{align*}
  Thus, a union bound argument yields
  \begin{align*}
    \P{B \geq t}
    & \leq \sum_{j} \P{B_j \geq t} 
     \leq \sum_{j} 2 \exp \bra{-\frac{\log \mN_n}{\alpha} t} \\
    & = 2 \mN_n \exp \bra{-\frac{\log \mN_n}{\alpha} t}.
  \end{align*}
  We therefore find
  \begin{align*}
    \E{B}
    & = \int_{0}^{\infty} \P{B \geq t} dt \\
    & \leq \int_{0}^{\alpha} 1 dt + \int_{\alpha}^{\infty} 2 \mN_n \exp \bra{-\frac{\log \mN_n}{\alpha} t} dt \\
    & = \alpha + 2 \mN_n \frac{\alpha}{\log \mN_n} \exp \bra{-\frac{\log \mN_n}{\alpha} \alpha} \\
    & = \bra{1 + \frac{2}{\log \mN_n}} \gamma \sqrt{a_n \mu_n \log \mN_n} 
    \leq 3 \sqrt{a_n \mu_n \log \mN_n},
  \end{align*}
  where the last inequality follows from $\log\mN_n\geq7.4$. 
  Similarly, 
  \begin{align*}
    \E{B^2}
    & = 2\int_{0}^{\infty} t\P{B \geq t} dt \\
    & \leq 2\int_{0}^{\alpha} t dt + \int_{\alpha}^{\infty} 4t \mN_n \exp \bra{-\frac{\log \mN_n}{\alpha} t} dt \\
    & = \alpha^2 + 4 \mN_n \bra{\frac{\alpha^2}{\log \mN_n} + \frac{\alpha^2}{\log^2 \mN_n}}\exp \bra{-\frac{\log \mN_n}{\alpha} \alpha}\\
    & \leq \bra{1 + \frac{8}{\log \mN_n}} \gamma^2 a_n \mu_n \log \mN_n 
    \leq 12 a_n \mu_n \log \mN_n.
  \end{align*}
  This completes the proof.
\end{proof}

\begin{lem}
  \label{lem.5.6}
  \begin{align*}
    \E{ \sum_{l=0}^{\mu_n-1} g'_{*, l}}
    \leq (1+\ep) \E{ \sum_{l=0}^{\mu_n-1} g_{*, l}} + \frac{24(1+\ep)^2}{\ep} F^2 a_n \log \mN_n.
  \end{align*}
\end{lem}

\begin{proof}
By definition we have
\begin{align*}
    & \abs{ \E{ \sum_{l=0}^{\mu_n-1} g'_{*, l}} - \E{ \sum_{l=0}^{\mu_n-1} g_{*, l}} } 
     \leq \E{\abs{ \sum_{l=0}^{\mu_n-1} \bra{g'_{*, l} - g_{*, l}} }} 
     \leq 2F \E{r^* B} .
  \end{align*}
  Using the Schwarz inequality, we obtain
  \begin{align*}
    \E{r^* B}
    &\leq\E{ \frac{1}{\mu_n} \sum_{l=0}^{\mu_n-1} \E{g'_{*, l} \mid \bra{X_t}_{t\geq 0} } }^{1/2} \E{B^2}^{1/2} + \sqrt{\frac{4F^2 a_n \log \mN_n}{\mu_n}} \E{B}\\
    & =  \E{ \sum_{l=0}^{\mu_n-1} g'_{*, l}}^{1/2} \bra{\frac{\E{B^2}}{\mu_n}}^{1/2} + 2F \sqrt{\frac{a_n \log \mN_n}{\mu_n}} \E{B}.
  \end{align*}
Consequently, 
  \begin{align*}
    & \abs{ \E{ \sum_{l=0}^{\mu_n-1} g'_{*, l}} - \E{ \sum_{l=0}^{\mu_n-1} g_{*, l}} } \\
    & \leq 2 \E{ \sum_{l=0}^{\mu_n-1} g'_{*, l}}^{1/2} \bra{\frac{F^2 \E{B^2}}{\mu_n}}^{1/2} + 4F^2 \sqrt{\frac{a_n \log \mN_n}{\mu_n}} \E{B}.
  \end{align*}
  Hence, using Eq.(C.4) of \cite{ref.SH}, we find
  \begin{align*}
    \E{ \sum_{l=0}^{\mu_n-1} g'_{*, l}}
    & \leq (1+\ep) \E{ \sum_{l=0}^{\mu_n-1} g_{*, l}} \\
    & \quad + (1+\ep) \cdot 4F^2 \sqrt{\frac{a_n \log \mN_n}{\mu_n}} \E{B} + \frac{(1+\ep)^2}{\ep} \frac{F^2 \E{B^2}}{\mu_n} \\
    & \leq (1+\ep) \E{ \sum_{l=0}^{\mu_n-1} g_{*, l}} + \frac{24(1+\ep)^2}{\ep} F^2 a_n \log \mN_n,
  \end{align*}
  where the last inequality follows from Lemma \ref{lem.bernstein_ineq} and $1\leq(1+\ep)/\ep$.
\end{proof}


\subsubsection{Step 3: Conclusion}

  Using Lemmas \ref{lem.5.4} and \ref{lem.5.6}, we deduce
  \begin{align*}
    \E{ \sum_{l=0}^{\mu_n-1} \tilde{g}'_{*, l}}
    & \leq \E{ \sum_{l=0}^{\mu_n-1} g'_{*, l}} + 4F^2 a_n \mu_n \bemix{a_n \De} \\
    & \leq (1+\ep) \E{ \sum_{l=0}^{\mu_n-1} g_{*, l}} 
     + \frac{24(1+\ep)^2}{\ep} F^2 a_n \log \mN_n + 4F^2 a_n \mu_n \bemix{a_n \De} \\
    & \leq (1+\ep) \E{ \sum_{l=0}^{\mu_n-1} \tilde{g}_{*, l}} \\
    &\quad+ \frac{24(1+\ep)^2}{\ep} F^2 a_n \log \mN_n + (2+\ep) \cdot 4F^2 a_n \mu_n \bemix{a_n \De}.
  \end{align*}
  Additionally, we define
  \begin{align*}
    \tilde{g}_{2, *, l} := \sum_{k \in H_{2,l}} ((f_{j^*}-f_0)(\X))^2, \\
    \tilde{g}'_{2, *, l} := \sum_{k \in H_{2,l}} ((f_{j^*}-f_0)(\X'))^2,
  \end{align*}
  and then we find in a similar way,
  \begin{align*}
    \E{ \sum_{l=0}^{\mu_n-1} \tilde{g}'_{2, *, l}}
    & \leq (1+\ep) \E{ \sum_{l=0}^{\mu_n-1} \tilde{g}_{2, *, l}} \\
    & \quad + \frac{24(1+\ep)^2}{\ep} F^2 a_n \log \mN_n + (2+\ep) \cdot 4F^2 a_n \mu_n \bemix{a_n \De}.
  \end{align*}
  Now, note that
 \begin{align*}
    \eR{f_{j^*}}&=\frac{1}{n}\E{ \sum_{l=0}^{\mu_n-1} \tilde{g}'_{*, l}+ \sum_{l=0}^{\mu_n-1} \tilde{g}'_{2, *, l}
    +\sum_{k=2a_n\mu_n}^{n-1}((f_{j^*}-f_0)(\X'))^2},\\
    \heR{f_{j^*}}&=\frac{1}{n}\E{ \sum_{l=0}^{\mu_n-1} \tilde{g}_{*, l}+ \sum_{l=0}^{\mu_n-1} \tilde{g}_{2, *, l}
    +\sum_{k=2a_n\mu_n}^{n-1}((f_{j^*}-f_0)(\X))^2}.
  \end{align*}
  To evaluate remaining terms with indices $k \in \seq{2a_n\mu_n, \dots, n-1}$, we use the following inequality:
  \begin{align*}
    \abs{ ((f_{j^*}-f_0)(\X))^2 - ((f_{j^*}-f_0)(\X'))^2 }
    & \leq 4F^2.
  \end{align*}
  Finally, we use the following inequalities:
  \begin{align*}
    \abs{\eR{\hf} - \eR{f_{j^*}}}
    & \leq 4F\de, \\
    \abs{\heR{\hf} - \heR{f_{j^*}}}
    & \leq 4F\de.
  \end{align*}
  Combining all the equations above, we obtain the desired result. \qed


\subsection{Proof of Lemma \ref{lem.oracle_ineq2}}\label{sec:proof-martingale}


As in the previous subsection, we can take a $\delta$-net $\{f_1,\dots,f_{\mN_n}\}$ of $\mF$ with respect to $\|\cdot\|_\infty$, and there is a random variable $j^*$ such that $|\hf-f_{j^*}| \leq \delta$. 

For $k=0,1,\dots,n-1$, define
\begin{align*}
  \I & := \frac{1}{\De} \int_{k\De}^{(k+1)\De} \bra{b^{\mathsf i}(X_s) - b^{\mathsf i}(\X)} ds, \\
  \Z' & := \frac{1}{\De} \int_{k\De}^{(k+1)\De} \Sigma^{\mathsf i}(\X) dw_s, \\
  \Z'' & := \frac{1}{\De} \int_{k\De}^{(k+1)\De} \bra{\Sigma^{\mathsf i}(X_s) - \Sigma^{\mathsf i}(\X)} dw_s,
\end{align*}
where $\Sigma^{\mathsf i} := (\Sigma_{\mathsf i1}, \dots, \Sigma_{\mathsf id})$, and define $\Z = \Z' + \Z''$.
Then, $\Y$ is decomposed as 
\begin{equation}\label{eq:Y-reg}
\Y = b^{\mathsf i}(\X) + \I + \Z,
\end{equation}
where $\Y$ is defined in Section \ref{sec.Estimator}. 
We follow similar arguments as in Parts (II) and (III) of the proof of \cite[Lemma 4]{ref.SH}, but there are several differences. 
We begin by an auxiliary result. 
\begin{lem}\label{lemma:gronwall}
  Define $A_t := \{X_t \in [0,1]^d\}$ for every $t\geq0$. 
  Let $C_{\text{Lip}} := 4\bra{C_{b}^2 + C_\Sigma^2} e^{4\bra{C_{b}'^2 + C_\Sigma'^2}}$.
  For any $k \in \seq{0,\dots,n-1}$ and $s \in [k\De,(k+1)\De]$, 
  \begin{align*}
    \E{\abs{X_s-\X}^2 \1_{A_{k\De}}} \leq C_{\text{Lip}} (s-k\De).
  \end{align*}
\end{lem}

\begin{proof}
  Using the It\^o isometry,
  \begin{align*}
    & \E{\abs{X_s-\X}^2 \1_{A_{k\De}}} \\
    & = \E{\sum_{i=1}^d \bra{ \int_{k\De}^s b^{\mathsf i} (X_u) du + \int_{k\De}^s \Sigma^{\mathsf i} (X_u) dw_u }^2 \1_{A_{k\De}}} \\
    & \leq 2\E{\sum_{i=1}^d \bra{ \int_{k\De}^s b^{\mathsf i} (X_u) \1_{A_{k\De}} du }^2 + \sum_{i=1}^d \int_{k\De}^s \Sigma^{\mathsf i} (X_u) \Sigma^{\mathsf i} (X_u)^\top \1_{A_{k\De}} du } \\
    & \leq 2\E{\int_{k\De}^s \sum_{i=1}^d \bra{ \bra{b^{\mathsf i} (X_u)}^2 + \Sigma^{\mathsf i} (X_u) \Sigma^{\mathsf i} (X_u)^\top } \1_{A_{k\De}} du },
  \end{align*}
  where the last inequality follows from the Schwarz inequality and $s-k\De \leq 1$.
  Additionally, for any $u \in [k\De,(k+1)\De]$,
  \begin{align*}
    \bra{b^{\mathsf i} (X_u)}^2
    & = \bra{b^{\mathsf i} (\X) + \bra{b^{\mathsf i} (X_u) - b^{\mathsf i} (\X)}}^2 \\
    & \leq 2\bra{b^{\mathsf i} (\X)}^2 + 2\bra{b^{\mathsf i} (X_u) - b^{\mathsf i} (\X)}^2
  \end{align*}
  and
  \begin{align*}
    \Sigma^{\mathsf i} (X_u) \Sigma^{\mathsf i} (X_u)^\top
    & = \sum_j \bra{\Sigma^{\mathsf ij} (\X) + \bra{\Sigma^{\mathsf ij} (X_u) - \Sigma^{\mathsf ij} (\X)}}^2 \\
    & \leq 2\sum_j \bra{\Sigma^{\mathsf ij} (\X)}^2 + 2\sum_j \bra{\Sigma^{\mathsf ij} (X_u) - \Sigma^{ij} (\X)}^2.
  \end{align*}
  Using these inequalities and Assumption \ref{ass.SDE1}, we conclude
  \begin{align*}
    & \E{\abs{X_s-\X}^2 \1_{A_{k\De}}} \\
    & \leq 4\E{\int_{k\De}^s C_{b}^2 + C_\Sigma^2 + \bra{C_{b}'^2 + C_\Sigma'^2} \abs{X_u - \X}^2 \1_{A_{k\De}} du } \\
    & = 4\bra{C_{b}^2 + C_\Sigma^2} (s-k\De) + 4\bra{C_{b}'^2 + C_\Sigma'^2} \int_{k\De}^s \E{\abs{X_u - \X}^2 \1_{A_{k\De}}} du.
  \end{align*}
  Therefore, we obtain the desired result using Gronwall's inequality and Assumption \ref{ass.SDE6}.
\end{proof}

The above lemma yields the following estimates for $I_{k\De}$ and $Z_{k\De}''$:
\begin{cor}
  \label{cor.I_ineq}
  $\E{\I^2 \1_{A_{k\De}}} \leq\frac{1}{2} C_{b}'^2 C_{\text{Lip}} \De$ for any $k \in \seq{0,\dots,n-1}$.
\end{cor}

\begin{proof}
Using the Schwarz inequality and Assumption \ref{ass.SDE1},
  \begin{align*}
    \E{\I^2 \1_{A_{k\De}}}
    & \leq \E{\frac{1}{\De} \int_{k\De}^{(k+1)\De} \bra{b^{\mathsf i} (X_s) - b^{\mathsf i} (\X)}^2 \1_{A_{k\De}} ds} \\
    & = \frac{1}{\De} \int_{k\De}^{(k+1)\De} \E{\bra{b^{\mathsf i} (X_s) - b^{\mathsf i} (\X)}^2 \1_{A_{k\De}}} ds \\
    & \leq \frac{1}{\De} C_{b}'^2 \int_{k\De}^{(k+1)\De} \E{\abs{X_s-\X}^2 \1_{A_{k\De}}} ds.
  \end{align*}
  Lemma \ref{lemma:gronwall} yields the desired bound. 
\end{proof}

\begin{cor}
  \label{cor.Z_ineq}
  $\E{\Z''^2 \1_{A_{k\De}}} \leq\frac{1}{2} C_\Sigma'^2 C_{\text{Lip}}$ for any $k \in \seq{0,\dots,n-1}$.
\end{cor}

\begin{proof}
  Using the It\^o isometry and Assumption \ref{ass.SDE1},
  \begin{align*}
    & \E{\Z''^2 \1_{A_{k\De}}} \\
    & = \E{\frac{1}{\De^2} \int_{k\De}^{(k+1)\De} \bra{\Sigma^{\mathsf i}(X_s) - \Sigma^{\mathsf i}(\X)} \bra{\Sigma^{\mathsf i}(X_s) - \Sigma^{\mathsf i}(\X)}^\top \1_{A_{k\De}} ds} \\
    & = \frac{1}{\De^2} \int_{k\De}^{(k+1)\De} \E{\bra{\Sigma^{\mathsf i}(X_s) - \Sigma^{\mathsf i}(\X)} \bra{\Sigma^{\mathsf i}(X_s) - \Sigma^{\mathsf i}(\X)}^\top \1_{A_{k\De}}} ds \\
    & \leq \frac{1}{\De^2} C_\Sigma'^2 \int_{k\De}^{(k+1)\De} \E{\abs{X_s-\X}^2 \1_{A_{k\De}}} ds. 
  \end{align*}
  Lemma \ref{lemma:gronwall} yields the desired bound. 
\end{proof}
Note that Corollaries \ref{cor.I_ineq} and \ref{cor.Z_ineq} particularly imply the square-integrability of $I_{k\De}\1_{A_{k\De}}$ and $Z_{k\De}\1_{A_{k\De}}$. 


Next we consider the following decomposition of $\heR{\hf}$:
\begin{lem}
  \label{lem.Phi_bar_ineq}
  For any non-random $\bf \in \mF$, 
  \begin{multline}\label{eq:basic-decomp}
  \heR{\hf}
  =\Psibfn + \heR{\bf}\\
  + 2 \E{ \nmean (\hf - \bf)(\X)\I}
  +2 \E{ \nmean (\hf-f_0)(\X)\Z},
  \end{multline}
  where
  $
    \Psibfn := \E{ \eQ{\hf} - \eQ{\bf} }.
  $
\end{lem}

\begin{proof}
A straightforward computation shows
  \begin{align*}
    (\Y)^2 - (f_0 (\X))^2 
    & = \bra{\Y - \hf (\X)}^2 - \bra{(\hf-f_0) (\X)}^2 \\
    & \qquad + 2 \hf(\X) (\Y - f_0 (\X)).
  \end{align*}
  Since $\supp (\hf) \subset [0,1]^d$ by assumption, we have
  \begin{align*}
    \hf(\X) (\Y - f_0 (\X))
    &= \hf(\X) (\Y - b^{\mathsf i} (\X)) \\
    &= \hf(\X) (\I + \Z).
  \end{align*}
  In a similar way, we deduce
  \begin{align*}
    (\Y)^2 - (f_0 (\X))^2
    & = \bra{\Y - \bf (\X)}^2 - \bra{(\bf-f_0) (\X)}^2 \\
    & \qquad + 2 \bf(\X) (\I + \Z).
  \end{align*}
  Moreover, since $(Z_{k\De}\1_{A_{k\De}})_{k=0}^{n-1}$ is a sequence of martingale differences,
  \[
  \E{\nmean \bf(\X)\Z}=\E{\nmean f_0(\X)\Z}=0.
  \]
  We obtain \eqref{eq:basic-decomp} by combining these identities.
\end{proof}


We proceed to bound the quantities on the right hand side of \eqref{eq:basic-decomp}. 
We bound the term involving $I_{k\De}$ by the following lemma.
\begin{lem}
  \label{lem.I_ineq0}
  Let $C_I := \frac{1}{2} C_{b}'^2 C_{\text{Lip}}$.
  For any $\ep>0$, 
  \begin{align*}
    \E{ \abs{\nmean (\hf-\bf)(\X) \I}}
    \leq \frac{\ep}{4}\heR{\hf} + \frac{\ep}{2}\heR{\bf}+\frac{3C_I}{2\ep} \De.
  \end{align*}
\end{lem}

\begin{proof}
For $(f,c)\in\{(\hf,\ep/2),(\bf,\ep)\}$, we have by the AM-GM inequality and Corollary \ref{cor.I_ineq}
  \begin{align*}
    & \E{\abs{\nmean (f-f_0)(\X) \I}} \\
    & \qquad \leq \E{\nmean \frac{1}{2} \bra{c \bra{(f-f_0)(\X)}^2 + c^{-1} \I^2 \1_{A_{k\De}}}} \\
    & \qquad \leq \frac{c}{2}\heR{f} + \frac{C_I}{2c}\De.
  \end{align*}
  Combining this with the decomposition $\hf-\bf=(\hf-f_0)-(\bf-f_0)$ and the triangle inequality, we obtain the desired result.
\end{proof}

To bound the term involving $Z_{k\De}$, we first consider the quantity obtained by replacing $\hf$ with $f_{j^*}$. The following result plays a key role to bound this quantity. 
\begin{lem}\label{lemma:sn}
Let $M=(M_t)_{t\geq0}$ be a continuous local martingale. Then, for any $T,y>0$,
\[
\E{\frac{y}{\sqrt{\langle M\rangle_T+y^2}}\exp\left\{\frac{M_T^2}{2(\langle M\rangle_T+y^2)}\right\}}\leq1.
\]
Moreover, if $\Ep[\sqrt{\langle M\rangle_T}]>0$, then
\[
\E{\exp\left\{\frac{M_T^2}{4(\langle M\rangle_T+(\Ep[\sqrt{\langle M\rangle_T}])^2)}\right\}}\leq\sqrt 2.
\]
\end{lem}

\begin{proof}
This is a consequence of Lemma 1.2 and Theorem 2.1 in \cite{ref.VictorH}.
\end{proof}

\begin{lem}
  \label{lem.Z_ineq2}
  \begin{align*}
    & \E{\abs{\nmean (f_{j^*}-f_0) (\X) \Z}} \\
    & \leq \bra{ \frac{4C_\Sigma^2}{n\De} \bra{2\log{\mN_n} + \log 2}}^{1/2} 
    \bra{\heR{\hf}+ 4F\de + \tfrac{2F^2 C_\Sigma'^2 C_{\text{Lip}}}{C_\Sigma^2}\De}^{1/2}. 
  \end{align*}
\end{lem}

\begin{proof}
  For any (possibly random) $f\in\mathcal F$, define processes $\hM(f)=(\hM(f)_s)_{s\geq0}$ and $\hA(f)=(\hA(f)_s)_{s\geq0}$ by
  \begin{align*}
    \hM(f)_s &:= \sum_{k=0}^{n-1} (f-f_0)(\X) \int_{s \wedge (k\De)}^{s \wedge ((k+1)\De)} \Sigma^{\mathsf i}(X_u) dw_u, \\
    \hA(f)_s &:= \sum_{k=0}^{n-1} ((f-f_0)(\X))^2 \int_{s \wedge (k\De)}^{s \wedge ((k+1)\De)} \abs{\Sigma^{\mathsf i}(X_u)}^2 du.
  \end{align*}
  Note that
  \begin{align*}
    \hM(f)_{n\De}
    & = n\De \cdot \nmean (f-f_0)(\X) \Z.
  \end{align*}
  Next, take $\varepsilon>0$ arbitrarily and set $D:=\Ep[\hat A(f_{j_*})_{n\Delta}]+\varepsilon$. Then we define
   \begin{align*}
    \xi_j &:= \frac{ \hM(f_j)_{n\De} }{2\sqrt{ \hA(f_j)_{n\De} + D }},\quad j=1,\dots,\mathcal N_n.
  \end{align*}
  For every $j$, $f_j$ is non-random and thus $\hM(f_j)$ is a continuous local martingale with $\langle\hM(f_j)\rangle=\hA(f_j)$. Hence, Lemma \ref{lemma:sn} yields
  \begin{align*}
   \E{\frac{\sqrt D}{\sqrt{\hA(f_j)_{n\De}+D}}\exp\left(2\xi_j^2\right)}\leq1.
  \end{align*}
  Meanwhile, we have by the Schwarz inequality
  \begin{align*}
   \E{\exp(\xi_{j^*}^2)}
   \leq\sqrt{\E{\frac{\sqrt D}{\sqrt{\hA(f_{j^*})_{n\De}+D}}\exp\left(2\xi_{j^*}^2\right)}\E{\frac{\sqrt{\hA(f_{j^*})_{n\De}+D}}{\sqrt D}}}.
  \end{align*}
  Since
  \begin{align*}
  \E{\frac{\sqrt D}{\sqrt{\hA(f_{j^*})_{n\De}+D}}\exp\left(2\xi_{j^*}^2\right)}
  &\leq\E{\max_{1\leq j\leq \mathcal N_n}\frac{\sqrt D}{\sqrt{\hA(f_{j})_{n\De}+D}}\exp\left(2\xi_{j}^2\right)}\\
  &\leq\mathcal N_n
  \end{align*}
  and
  \begin{align*}
   \E{\frac{\sqrt{\hA(f_{j^*})_{n\De}+D}}{\sqrt D}}
   \leq\sqrt{\E{\frac{\hA(f_{j^*})_{n\De}+D}{D}}}
   \leq\sqrt 2,
  \end{align*}
  we obtain
  \begin{align*}
   \E{\exp(\xi_{j^*}^2)}
   \leq2^{1/4}\sqrt{\mathcal N_n}.
  \end{align*}
  Hence, the Jensen inequality yields
  \[
  \E{\xi_{j^*}^2}\leq\log\E{\exp(\xi_{j^*}^2)}
  \leq\frac{1}{4}\log 2+\frac{1}{2}\log\mathcal N_n.
  \] 
  Furthermore, 
  \begin{align*}
    \E{\hA(f_{j^*})_{n\De}}
    & \leq 2C_\Sigma^2 n\De \heR{f_{j^*}} + 8F^2 \De^2 \sum_{k=0}^{n-1}\E{\Z''^2 \1_{A_{k\De}}}, \\
    & \leq 2C_\Sigma^2 n\De \bra{\heR{\hf} + 4F\de} + 4F^2 C_\Sigma'^2 C_{\text{Lip}} n\De^2,
  \end{align*}
  where the first inequality follows from Assumption \ref{ass.SDE1} and $\abs{\Sigma^{\mathsf i}(X_u)}^2 \leq 2\abs{\Sigma^{\mathsf i}(\X)}^2 + 2\abs{\Sigma^{\mathsf i}(X_u) - \Sigma^{\mathsf i}(\X)}^2$, and the last inequality follows from Corollary \ref{cor.Z_ineq}.
  Combining these bounds with the Schwarz inequality, we deduce
  \begin{align*}
     &\E{\abs{\frac{1}{n\De} \hM(f_{j^*})_{n\De}}} 
     = \frac{2}{n\De} \E{ \abs{\xi_{j^*}} \bra{\hA(f_{j^*})_{n\De} + D}^{1/2} } \\
    & \leq \frac{2}{n\De} \sqrt{\E{\xi_{j^*}^2} \bra{2\E{\hA(f_{j^*})_{n\De}} +\varepsilon}}\\
    &\leq\sqrt{\frac{\log 2+2\log\mathcal N_n}{(n\De)^2}\bra{4C_\Sigma^2 n\De \bra{\heR{\hf} + 4F\de} + 8F^2 C_\Sigma'^2 C_{\text{Lip}} n\De^2 +\varepsilon}}.
  \end{align*}
  Letting $\varepsilon$ tend to 0, we obtain the desired result. 
\end{proof}

Next we evaluate the approximation error induced by replacing $\hf$ with $f_{j_*}$. 
\begin{lem}
  \label{lem.Z_ineq0}
  \begin{align*}
    \E{\abs{\nmean (\hf-f_{j^*})(\X) \Z''}}
    & \leq \de \sqrt{\frac{1}{2} C_\Sigma'^2 C_{\text{Lip}}}.
  \end{align*}
\end{lem}

\begin{proof}
We have
  \begin{align*}
    \E{\abs{\nmean (\hf-f_{j^*})(\X) \Z''}} \leq \nmean \E{\abs{(\hf - f_{j^*})(\X)} \abs{\Z'' \1_{A_{k\De}}}}.
  \end{align*}
  Therefore, we obtain the desired result using $\norm{\hf - f_{j^*}}_\infty \leq \de$ and Corollary \ref{cor.Z_ineq} with $\E{\abs{\Z''} \1_{A_{k\De}}}^2 \leq \E{\Z''^2 \1_{A_{k\De}}}$.
\end{proof}

\begin{lem}
  \label{lem.Z_ineq}
  There exists a constant $C_Z$ depending only on $C_\Sigma$ such that
  \begin{align*}
    \E{\abs{\nmean (\hf-f_{j^*})(\X) \Z'}}
    & \leq \frac{C_Z}{\sqrt{n\De}} \int_0^\de \sqrt{\log \mN (u, \mF, \norm{\cdot}_\infty)} du.
  \end{align*}
\end{lem}

\begin{proof}
  For non-random $f, g \in \mF$, we consider the following processes:
  \begin{align*}
    \tM(f,g)_s &:= \frac{1}{2C_\Sigma \sqrt{n\De}} \sum_{k=0}^{n-1} (f-g)(\X) \int_{s \wedge (k\De)}^{s \wedge ((k+1)\De)} \Sigma^{\mathsf i}(\X) dw_u, \\
    \tA(f,g)_s &:= \frac{1}{4C_\Sigma^2 n\Delta} \sum_{k=0}^{n-1} ((f-g)(\X))^2 \int_{s \wedge (k\De)}^{s \wedge ((k+1)\De)} \abs{\Sigma^{\mathsf i} (\X)}^2 du.
  \end{align*}
  Then, $\tM(f,g)_s$ is a continuous local martingale and $\tA(f,g)_s$ gives its quadratic variation. 
  Therefore, if $\Ep[\tA(f,g)_{n\De}]>0$, Lemma \ref{lemma:sn} along with the inequality $(\Ep[\tA(f,g)_{n\De}^{1/2}])^2\leq\Ep[\tA(f,g)_{n\De}]$ imply that
  \begin{align*}
    \E{\exp \bra{\frac{\tM(f,g)_{n\De}^2}{4 \bra{\tA(f,g)_{n\De} + \Ep[\tA(f,g)_{n\De}]}}}} \leq \sqrt{2}.
  \end{align*}
  Thus, for any $x>0$,
  \begin{align*}
    & \P{\abs{\tM(f,g)_{n\De}} \geq x } \\
    & \qquad \leq \P{\exp \bra{\frac{\tM(f,g)_{n\De}^2}{4 \bra{\tA(f,g)_{n\De} + \Ep[\tA(f,g)_{n\De}]}}} \geq \exp \bra{\frac{x^2}{2\norm{f-g}_\infty^2}} } \\
    & \qquad \leq \exp \bra{-\frac{x^2}{2\norm{f-g}_\infty^2}} \E{\exp \bra{\frac{\tM(f,g)_{n\De}^2}{4 \bra{\tA(f,g)_{n\De} + \Ep[\tA(f,g)_{n\De}]}}}} \\
    & \qquad \leq 2\exp \bra{-\frac{x^2}{2\norm{f-g}_\infty^2}},
  \end{align*}
  where we used $\tA(f,g)_{n\De} \leq \tfrac{1}{4} \norm{f-g}_\infty^2$ for the first inequality. 
  When $\Ep[\tA(f,g)_{n\De}]=0$, then $\tM(f,g)_{n\De}=0$ and hence the above inequality is evident. 
  Now, for each $f\in\mF$, define
  \begin{align*}
    \tZ'(f) &:= \frac{1}{2C_\Sigma \sqrt{n\De}} \sum_{k=0}^{n-1} f(\X) \int_{k\De}^{(k+1)\De} \Sigma^{\mathsf i} (\X) dw_u.
  \end{align*}
  Then we have $\tZ'(f) - \tZ'(g) = \tM(f,g)_{n\De}$ for any $f, g \in \mF$, so the process $(\tZ'(f))_{f \in \mF}$ is sub-Gaussian with respect to $\norm{\cdot}_\infty$ (cf. Section 2.1.3 of \cite{ref.Bernstein_ineq}). 
  Hence, Corollary 2.2.8 in \cite{ref.Bernstein_ineq} yields
  \begin{align*}
     \E{\abs{\nmean (\hf-f_{j^*})(\X) \Z'}} 
    &  \leq \frac{2C_\Sigma}{\sqrt{n\De}} \E{ \sup_{\norm{f-g}_\infty \leq \de} \abs{\tZ'(f) - \tZ'(g)}} \\
    &  \leq \frac{2C_\Sigma K}{\sqrt{n\De}} \int_0^\de \sqrt{\log \mN (u, \mF, \norm{\cdot}_\infty)} du,
  \end{align*}
  where $K>0$ is a universal constant. 
\end{proof}

We are now ready to prove Lemma \ref{lem.oracle_ineq2}. 
\begin{proof}[Proof of Lemma \ref{lem.oracle_ineq2}]
  We have by Lemma \ref{lem.Z_ineq2} and the AM-GM inequality 
  \begin{align*}
    \E{\abs{\nmean (f_{j^*}-f_0)(\X) \Z}}
    & \leq \frac{\ep}{4} \heR{\hf} + \gamma_\ep,
  \end{align*}
  where 
  \[
  \gamma_\ep = \frac{\ep}{4} F\de + \frac{\ep F^2 C_\Sigma'^2 C_{\text{Lip}}}{2C_\Sigma^2} \De + \frac{4C_\Sigma^2}{\ep n\De} \bra{2\log{\mN_n} + \log2 }.
  \]
  Combining this with Lemmas \ref{lem.Phi_bar_ineq}, \ref{lem.I_ineq0} and \ref{lem.Z_ineq2}--\ref{lem.Z_ineq}, we obtain for any $\bf \in \mF$
  \begin{align*}
    \heR{\hf}
    & \leq \PsiFn{\hf} + \heR{\bf} + \frac{\ep}{2} \heR{\hf} \\
    & \quad  + \ep \heR{\bf} + \frac{3}{\ep} C_I F^2 \De + \sqrt{2 C_\Sigma'^2 C_{\text{Lip}}} \de\\
    & \quad  + \frac{2C_Z}{\sqrt{n\De}} \int_0^\de \sqrt{\log \mN (u, \mF, \norm{\cdot}_\infty)} du 
     + \frac{\ep}{2} \heR{\hf} + 2\gamma_\ep.
  \end{align*}
  Therefore, noting $\heR{\bf}=\eR{\bf}$, we obtain
  \begin{align*}
    (1-\ep) \heR{\hf}
    \leq \PsiFn{\hf} + (1+\ep)\heR{\bf} + \gamma'_\ep,
  \end{align*}
  where 
  \[
  \gamma'_\ep = \frac{3}{\ep} C_I F^2 \De + \sqrt{2 C_\Sigma'^2 C_{\text{Lip}}} \de + \frac{2C_Z}{\sqrt{n\De}} \int_0^\de \sqrt{\log \mN (u, \mF, \norm{\cdot}_\infty)} du + 2\gamma_\ep.
  \]
  Taking the infimum over $\bf$, we conclude
  \begin{align*}
    \heR{\hf}
    \leq \frac{1}{1-\ep}\PsiFn{\hf} + \frac{1+\ep}{1-\ep}\inf_{f\in\mathcal F}\eR{f} + \frac{1}{1-\ep}\gamma'_\ep.
  \end{align*}
  Noting $F\geq1$, we obtain the desired result. 
\end{proof}




\subsection{Proof of Theorem \ref{thm.oracle_ineq}}\label{sec:proof-oracle_ineq}

We can estimate the covering number of a subset of $\mF(L, \bp, s)$ as follows. 
\begin{lem}
  \label{lem.entropy}
Let $\mF \subset \mF(L, \bp, s)$. If $s\geq2$, we have for all $\delta\in(0,1)$
\begin{align*}
    \log \mathcal{N} \bra{ \delta, \mF, \|\cdot \|_\infty }
    \leq Cs(L\log s + \log d - \log \delta),
  \end{align*}
  where $C>0$ is a universal constant. 
\end{lem}
\begin{proof}
  Note that $\mN (\delta, \mF, \|\cdot \|_\infty) \leq \mN (\tfrac{1}{2} \delta, \mF(L, \bp, s), \|\cdot \|_\infty)$ by Exercise 4.2.9 of \cite{ref.Vershynin}. Thus, we have by Lemma 5 and Remark 1 in \cite{ref.SH}
  \begin{align*}
  \log\mN (\delta, \mF, \|\cdot \|_\infty)\leq(s+1)\log\bra{2^{2L+6}\delta^{-1}(L+1)d^2s^{2L}}.
  \end{align*}
  Then, using $\log s\geq\log2$ and $\log(L+1)\leq L$, we obtain the desired result. 
\end{proof}

\begin{proof}[Proof of Theorem \ref{thm.oracle_ineq}]
%
We apply Lemmas \ref{lem.oracle_ineq} and \ref{lem.oracle_ineq2} with $\de = (n\De)^{-1}$ and $a_n = \frac{1}{C_{\be} \De} \log (n\De)$. We may assume $\mu_n:=\lfloor n/2a_n\rfloor>0$ without loss of generality; otherwise, $n< 2a_n$ and thus $\frac{\log(n\De)}{n\De}>C_\beta/2$, so the desired bound holds with $C_\tau=2/C_\be$. 
Hence, noting $\De\leq1$ and $F\geq1$, we obtain by Lemmas \ref{lem.oracle_ineq} and \ref{lem.oracle_ineq2}
\begin{multline*}
\eR{f}\leq\tau\bra{\PsiFn{\hf} + \inf_{f\in \mF} \eR{f} }
+c_1 F^2\left(\frac{\log\mathcal N_n}{\mu_n}+\beta_X(a_n\De)\right.\\
\left.+\frac{\log\mathcal N_n}{n\De}
+\De+\frac{1}{\sqrt{n\De}} \int_0^\de \sqrt{\log \mN (u, \mF, \norm{\cdot}_\infty)} du
\right),
\end{multline*}
where $\mF:=\mF(L, \bp, s, F)$, $\mathcal N_n:=\mN (\delta, \mF, \norm{\cdot}_\infty)\vee\lceil\exp(7.4)\rceil$ and $c_1>0$ is a constant depending only on $(C_{b}, C_{b}', C_\Sigma, C_\Sigma', \tau)$. 
Now, Lemma \ref{lem.entropy} and Assumption \ref{ass.SDE6} yield
\[
\frac{\log\mathcal N_n}{\mu_n}+\frac{\log\mathcal N_n}{n\De}
\leq c_2\frac{s(L\log s +\log(n\De))\cdot\log(n\De)}{n\De},
\]
where $c_2>0$ depends only on $C_\beta$. 
Also, Assumption \ref{ass.mixing} yields $\beta_X(a_n\De)\leq C_\beta'(n\De)^{-1}$. 
Moreover, we have by Lemma \ref{lem.entropy}
\begin{align*}
  & \int_0^{(n\De)^{-1}} \sqrt{\log \mN (u, \mF, \norm{\cdot}_\infty)} du \\
  & \qquad \leq c_3\int_0^{(n\De)^{-1}} \sqrt{s(L\log s + \log u^{-1})} du 
   = c_3\int_{n\De}^\infty \frac{1}{r^2} \sqrt{s(L\log s + \log r)} dr \\
  & \qquad \leq c_3\int_{n\De}^\infty \bra{ \frac{1}{r^2} \sqrt{s(L\log s + \log r)} - \frac{s}{2r^2 \sqrt{s(L\log s + \log r)}} } dr + c_3\int_{n\De}^\infty \frac{\sqrt{s}}{r^2} dr \\
  & \qquad = c_3\left[ -r^{-1} \sqrt{s(L\log s + \log r)} \right]_{n\De}^\infty + c_3\left[ -\sqrt{s} r^{-1} \right]_{n\De}^\infty \\
  & \qquad = c_3\frac{\sqrt{s(L\log s + \log (n\De))} + \sqrt{s}}{n\De},
\end{align*}
where $c_3>0$ is a universal constant. Finally, note that $\eR{f} \leq \norm{f - f_0}_\infty^2$. Combining all the estimates above, we obtain the desired result. 
\end{proof}


\section{Proof of Theorem \ref{thm.main}}
\label{sec.proof_main}

The following result is essentially shown in the proof of Theorem 1 in \cite{ref.SH}. 
\begin{lem}\label{lem:approx}
For any $f\in\mG\big(q, \bd, \bt, \bbeta, K\big)$ there are a function $\widetilde f_n^* \in \mF(L, \bp, s, F)$ and a constant $C>0$ depending only on $(c_L^u, c_{\bp}, c_s^l, c_s^u, q, \bd, \bt, \bbeta, F)$ such that $\|\widetilde f_n^* - f_0\|_\infty^2 \leq C\phi_n$.
\end{lem}
We give a proof of this lemma in Section \ref{proof:lem:approx} for the sake of completenss.

\begin{proof}[Proof of Theorem \ref{thm.main}]
Assumption \ref{ass.SDE6} and condition (iv) imply that
\[
\frac{s (L\log s + \log (n\De)) \cdot \log (n\De)}{n\De}\leq c\phi_nL\log^3(n\De),
\]
where $c>$ depends only on $(c_s^u,q,\bt,\bbeta)$. Combining this with Theorem \ref{thm.oracle_ineq} and Lemma \ref{lem:approx}, we obtain the desired result. 
\end{proof}

\section{Proof of Lemma \ref{lemma:beta}}\label{sec:lemma-beta}

We begin by introducing some general notation. Given a Markov kernel $\mathcal P$ on $\mathbb R^d$ and a probability measure $\mu$ on $\mathbb R^d$, we define the probability measure $\mu\mathcal P$ on $\mathbb R^d$ by $(\mu\mathcal P)(\cdot)=\int_{\mathbb R^d}\mathcal P(x,\cdot)\mu(dx)$. 
The total variation measure of a signed measure $\nu$ is denoted by $|\nu|$. 
For $x\in\mathbb R^d$, $\delta_x$ denotes the Dirac measure at $x$. 

For the proof, we apply Theorem 1.3 in \cite{HaMa11}. 
Let $(P_b^t)_{t\geq0}$ be the transition semigroup associated with the SDE \eqref{eq:sde}. 
First we check Assumption 1 in \cite{HaMa11} (geometric drift condition). Fix a $C^\infty$ function $\kappa:\mathbb R\to[0,1]$ such that $\kappa(x)=0$ for $x\leq1$ and $\kappa(x)=1$ for $x\geq2$. Define the function $V:\mathbb R^d\to(0,\infty)$ by $V(x)=\exp(r|x|\kappa(|x|))$ for $x\in\mathbb R^d$. $V$ is evidently a $C^\infty$ function. Moreover, let $\mathcal L$ be the generator associated with the SDE \eqref{eq:sde}. Then, for any $x\in\mathbb R^d$ with $|x|>2d$, we have
\begin{align*}
\mathcal LV(x)&=r V(x)\left\{\frac{\langle b(x),x\rangle}{|x|}+\frac{r}{2}\bra{r+\frac{d-1}{|x|}}\right\}\\
&\leq rV(x)\left\{-r+\frac{r}{2}\bra{r+\frac{1}{2}}\right\}
\leq-\frac{r^2}{4}V(x).
\end{align*}
Therefore, setting $a:=-r^2/4$ and 
\[
C_0:=\bra{1+a}\sup_{x\in\mathbb R^d:|x|\leq2d}\bra{\sqrt{d}F|\nabla V(x)|+\frac{1}{2}|\nabla\cdot\nabla V(x)|}<\infty,
\]
we obtain
\[
\mathcal LV(x)\leq-a V(x)+C_0\qquad\text{for any }x\in\mathbb R^d.
\]
Hence, we deduce from Example 3.2.5 in \cite{Ku18}
\begin{equation}\label{eq:lyapunov}
\int_{\mathbb R^d}V(y)P_{b}^t(x,dy)\leq e^{-at}V(x)+C_0\frac{1-e^{-at}}{a}
\end{equation}
for any $x\in\mathbb R^d$ and $t>0$.
Next we check Assumption 2 in \cite{HaMa11} (minorization condition). Set 
\[
R:=\frac{3C_0}{a}\qquad\text{and}\qquad
\mathcal C:=\{x\in\mathbb R^d:V(x)\leq R\}.
\]
Note that $\mathcal C$ is compact. 
Theorem 2 in \cite{QiRuZh03} implies that, for any $x,y\in\mathbb R^d$,
\begin{align*}
p_b^1(x,y)\geq\frac{1}{(2\pi)^{d/2}}e^{-\frac{(|x-y|+\|b\|)^2}{2}-\|b\|(\zeta_d+|x-y|)},
\end{align*}
where $p_b^1(x,\cdot)$ is a density of $P_b^1(x,\cdot)$, $\|b\|:=\sup_{x\in\mathbb R^d}|b(x)|$ and $\zeta_d$ is a positive constant depending only on $d$. Since
\begin{align*}
&\frac{(|x-y|+\|b\|)^2}{2}+\|b\|(\zeta_d+|x-y|)\\
&=\frac{|x-y|^2+\|b\|^2}{2}+\|b\|\zeta_d+2\|b\||x-y|\\
&\leq|x-y|^2+\|b\|^2\bra{\frac{5}{2}+\zeta_d}\leq 2|x|^2+2|y|^2+dK^2(3+\zeta_d),
\end{align*}
we obtain
\begin{equation*}
p_b^1(x,y)\geq\varphi(y)\cdot2^{-d/2}e^{-2|x|^2-dK^2(3+\zeta_d)},
\end{equation*}
where $\varphi$ is the density of the $d$-dimensional normal distribution with mean 0 and covariance matrix $2^{-1}I_d$. Therefore, letting
\[
\alpha:=\inf_{x\in\mathcal C}2^{-d/2}e^{-2|x|^2-dK^2(3+\zeta_d)}\in(0,1),
\]
we obtain
\begin{equation}\label{eq:minorization}
\inf_{x\in\mathcal C}p_b^1(x,y)\geq\alpha\varphi(y)\qquad\text{for any }y\in\mathbb R^d,
\end{equation}
which yields Assumption 2 in \cite{HaMa11}. Consequently, we have by Theorem 1.3 in \cite{HaMa11}
\begin{equation}\label{eq:contraction}
\rho_\beta(\mu_1P_b^1,\mu_2P_b^1)\leq\bar\alpha\rho_\beta(\mu_1,\mu_2)
\end{equation}
for any probability measures $\mu_1$ and $\mu_2$ on $\mathbb R^d$, where $\beta>0$ and $\bar\alpha\in(0,1)$ depend only on $r,d,K$ (and $\kappa$), and
\[
\rho_\beta(\mu_1,\mu_2):=\int_{\mathbb R^d}(1+\beta V(x))|\mu_1-\mu_2|(dx).
\]
Now, applying \eqref{eq:contraction} repeatedly, we obtain
\begin{equation}\label{eq:geom}
\rho_\beta(\mu_1P_b^n,\mu_2P_b^n)\leq\bar\alpha^n\rho_\beta(\mu_1,\mu_2)\qquad\text{for }n=1,2,\dots.
\end{equation}
Therefore, for any integer $n\geq1$,
\begin{align*}
\beta_X(n)&=\sup_{t\geq0}\int_{\mathbb R^d}\|\delta_xP_b^n-\mathcal L_\eta P_b^{t+n}\|\mathcal L_\eta P_b^t(dx)\\
&\leq\sup_{t\geq0}\int_{\mathbb R^d}\rho_\beta(\delta_xP_b^n,(\mathcal L_\eta P_b^t)P_b^{n})\mathcal L_\eta P_b^t(dx)\\
&\leq\bar\alpha^n\sup_{t\geq0}\int_{\mathbb R^d}\rho_\beta(\delta_x,\mathcal L_\eta P_b^t)\mathcal L_\eta P_b^t(dx)\\
&\leq2\bar\alpha^n\sup_{t\geq0}\bra{1+\beta \int_{\mathbb R^d}V(x)\mathcal L_\eta P_b^t(dx)}\\
&\leq2\bar\alpha^n\bra{1+\beta \int_{\mathbb R^d}V(x)\mathcal L_\eta (dx)+\frac{C_0}{a}},
\end{align*}
where the first equality follows from \cite[Proposition 1]{Da73} and the last inequality follows from \eqref{eq:lyapunov}, respectively. Finally, note that $\beta_X(t)\leq\beta_X(\lfloor t\rfloor)$ for any $t\geq0$. Thus we complete the proof of \eqref{eq:minimax-beta}. \qed

\section{Proof of Theorem \ref{thm:minimax}}\label{sec:proof-minimax}

Throughout this section, we write $\mathcal B=\mathcal{B}^{\mathsf i}(r, q, \bd, \bt, \bbeta, K)$ and $\|\cdot\|_2=\|\cdot\|_{L^2([0,1]^d)}$ for short. 
We begin by reducing the problem to establishing a lower bound on the minimax $L^2$-estimation error. 
\begin{lem}\label{lem:reduce-to-l2}
There is a constant $c>0$ such that
\begin{equation}\label{reduce-to-l2}
\liminf_{n\to\infty}\inf_{\hf}\sup_{b\in\mathcal B}\phi_n^{-1}\eR{\hf}
\geq c\liminf_{n\to\infty}\inf_{\hf}\sup_{b\in\mathcal B}\phi_n^{-1}\E{\|\hf-f_0\|_{2}^2}.
\end{equation}
\end{lem}

\begin{proof}
Throughout the proof, we will use the same notation as in Section \ref{sec:lemma-beta}. 
First, from the proof of Lemma \ref{lemma:beta} and \cite[Theorem 3.2]{HaMa11}, $P_b^1$ has the invariant distribution $\Pi_b$ for all $b\in\mathcal B_0(d,r,K)$. 
Next, fix an estimator $\hf$ arbitrarily. Define $\tilde f_n:=\{(-K)\vee\hf\}\wedge K\1_{[0,1]^d}$. Since $\|f_0\|_\infty\leq K$ and $\supp (f_0) \subset[0,1]^d$, we have $|\tilde f_n-f_0|\leq|\hf-f_0|$. Hence $\eR{\hf}\geq\eR{\tilde f_n}$
Since
\[
\eR{\tilde f_n}=\nmean\E{\int_{\mathbb R^d}\left\{\int_{\mathbb R^d}|\tilde f_n(y)-f_0(y)|^2P_b^{k\De}(x,dy)\right\}\mathcal L_\eta(dx)},
\]
we have
\begin{align*}
&\left|\eR{\tf}-\E{\int_{\mathbb R^d}|\tf(y)-f_0(y)|^2\Pi_b(dy)}\right|\\
&\leq4K^2\nmean\int_{\mathbb R^d}\|\delta_xP_b^{k\Delta}-\Pi_b\|\mathcal L_\eta(dx)
\leq4K^2\nmean\bar\alpha^{\lfloor k\Delta\rfloor}\int_{\mathbb R^d}\rho_\beta(\delta_x,\Pi_b)\mathcal L_\eta(dx),
\end{align*}
where the last inequality follows from \eqref{eq:geom}. We have
\[
\nmean\bar\alpha^{\lfloor k\Delta\rfloor}
\leq\frac{1}{n\bar\alpha}\sum_{k=0}^{n-1}\bar\alpha^{k\Delta}
\leq\frac{1}{n\bar\alpha(1-\bar\alpha^\Delta)}
\]
and 
\[
\int_{\mathbb R^d}\rho_\beta(\delta_x,\Pi_b)\mathcal L_\eta(dx)
\leq2+\beta\int_{\mathbb R^d}V(x)\mathcal L_\eta(dx)+\beta\int_{\mathbb R^d}V(x)\Pi_b(dx).
\]
One can easily derive the following estimate from \eqref{eq:lyapunov} (cf.~\cite[Proposition 4.24]{Ha18}):
\begin{equation}\label{eq:V-mom}
\int_{\mathbb R^d}V(x)\Pi_b(dx)\leq\frac{C_0}{a}. 
\end{equation}
Combining these estimates, we obtain
\[
\left|\eR{\tf}-\E{\int_{\mathbb R^d}|\tf(y)-f_0(y)|^2\Pi_b(dy)}\right|\leq \frac{C_1}{n(1-\bar\alpha^\Delta)},
\]
where $C_1$ is a constant depending only on $d,r,K$ and $\eta$. Since $(1-\bar\alpha^\Delta)/\Delta\to-\log\bar\alpha$ and $\phi_n^{-1}=o(n\Delta)$ as $n\to\infty$, we conclude
\besn{\label{risk-to-l1}
&\liminf_{n\to\infty}\inf_{\hf}\phi_n^{-1}\sup_{b\in\mathcal B}\eR{\hf}\\
&\geq\liminf_{n\to\infty}\inf_{\hf}\phi_n^{-1}\sup_{b\in\mathcal B}\E{\int_{\mathbb R^d}|\hf(y)-f_0(y)|^2\Pi_b(dy)}\\
&\geq\liminf_{n\to\infty}\inf_{\hf}\phi_n^{-1}\sup_{b\in\mathcal B}\E{\int_{[0,1]^d}|\hf(y)-f_0(y)|^2\Pi_b(dy)}.
}
Now, using the definition of $\Pi_b$, we can easily check that $\Pi_b$ has a density given by
\[
\pi_b(y)=\int_{\mathbb R^d}p_b^1(x,y)\Pi_b(dx),\qquad y\in\mathbb R^d.
\]
We have by \eqref{eq:minorization}
\[
\inf_{y\in[0,1]^d}\pi_b(y)\geq\alpha\inf_{y\in[0,1]^d}\varphi(y)\Pi_b(\mathcal C).
\]
By the Markov inequality and \eqref{eq:V-mom}, we obtain
\begin{align*}
1-\Pi_b(\mathcal C)
=\Pi_b(V>R)
\leq\frac{1}{R}\int_{\mathbb R^d}V(x)\Pi_b(dx)
\leq\frac{1}{3}.
\end{align*}
Hence we conclude
\[
\inf_{y\in[0,1]^d}\pi_b(y)\geq\frac{\alpha}{3}\inf_{y\in[0,1]^d}\varphi(y).
\]
Consequently, there is a constant $c>0$ depending only on $d,r$ and $K$ such that
\begin{equation}\label{pi-bound}
\E{\int_{[0,1]^d}|\hf(y)-f_0(y)|^2\Pi_b(dy)}\geq c\E{\int_{[0,1]^d}|\hf(y)-f_0(y)|^2dy}.
\end{equation}
Combining \eqref{risk-to-l1} with \eqref{pi-bound}, we obtain the desired result.
\end{proof}

To establish a lower bound on the right hand side of \eqref{reduce-to-l2}, we employ the functions in $\ol\mG\big(q, \bd, \bt, \bbeta, K\big)$ constructed in the proof of \cite[Theorem 3]{ref.SH}. Properties of these functions are summarized as follows:
\begin{lem}\label{lem:minimax}
There is a constant $K>0$ having the following property: For any $n\in\mathbb N$ and $\Gamma>0$, there are an integer $M\geq1$ and functions $f_{(0)},\dots,f_{(M)}\in\ol\mG\big(q, \bd, \bt, \bbeta, K\big)$ satisfying the following conditions:
\begin{enumerate}[label=(\roman*)]

\item\label{f-l2} For all $0\leq j<k\leq M$, 
\begin{equation}\label{f-upper}
n\De\|f_{(j)}-f_{(k)}\|_2^2\leq\frac{\log M}{9\Gamma}
\end{equation}
and
\begin{equation}\label{f-lower}
\|f_{(j)}-f_{(k)}\|_2^2\geq\kappa\phi_n,
\end{equation}
where $\kappa>0$ is a constant depending only on $\bt$ and $\bbeta$. 

\item\label{f-smooth} $|f_{(j)}(x)-f_{(j)}(y)|\leq K|x-y|^{\beta^*}$ for all $j=0,1,\dots,M$ and $x,y\in[0,1]^d$. 


\item\label{f-bound} $\|f_{(j)}\|_\infty\leq (2(\beta^*)^{t^*}t^*)^B$ for all $j=0,1,\dots,M$. 

\end{enumerate}
\end{lem}

\begin{proof}[Proof of Theorem \ref{thm:minimax}]
For each $b\in \mathcal B$ and $n\in\mathbb N$, we denote by $P_{b,n}$ the law of $\mathbb X_n:=(X_t)_{t\in[0,n\De]}$ induced on $C([0,n\De],\mathbb R^d)$ when $X$ is defined by \eqref{eq:sde}. Here, $C([0,n\De],\mathbb R^d)$ is the space of all continuous functions from $[0,n\De]$ to $\mathbb R^d$ equipped with the supremum norm. 
Moreover, $E_{b,n}[\cdot]$ denotes the expectation under $P_{b,n}$. 

Since the continuous observation model is more informative than the discrete observation model, the right hand side of \eqref{reduce-to-l2} is bounded from below by 
\[
c\liminf_{n\to\infty}\inf_{\hat b^{\mathsf i}_{n\De}}\sup_{b\in\mathcal B}\E{\|\hat b^{\mathsf i}_{n\De} -b^{\mathsf i}\|_{2}^2},
\] 
where the infimum is taken over all estimators $\hat b^{\mathsf i}_{n\De}$ based on the continuous observations $\mathbb X_n$, i.e.~all real-valued random functions $\hat b^{\mathsf i}_{n\De}$ defined on $\mathbb R^d$ such that $(\omega,x)\mapsto\hat b^{\mathsf i}_{n\De}(\omega,x)$ is measurable with respect to the product of the $\sigma$-algebra generated by $\mathbb X_n$ and the Borel $\sigma$-algebra of $\R^d$. 
Therefore, according to Theorem 2.7 in \cite{Ts09} and Lemma \ref{lem:reduce-to-l2}, it suffices to show that there is a constant $K>0$ having the following property: For every $n\in\mathbb N$, there are an integer $M\geq1$ and functions $b_0,b_1,\dots,b_M\in\mathcal{B}$ such that
\begin{equation}\label{tsy-eq1}
\|b^{\mathsf i}_{j}-b^{\mathsf i}_{k}\|_2^2\geq \kappa\phi_n\quad\text{for all }0\leq j<k\leq M
\end{equation}
and
\begin{equation}\label{tsy-eq2}
P_{j}\ll P_{0}\quad\text{for all }j=1,\dots,M
\end{equation}
and
\begin{equation}\label{tsy-eq3}
\frac{1}{M}\sum_{j=1}^M\int\log\frac{dP_{j}}{dP_{0}}dP_j\leq \frac{1}{9}\log M,
\end{equation}
where $\kappa>0$ is a constant independent of $n$ and $P_j:=P_{b_j,n}$ for $j=0,1,\dots,M$. 

Set $B:=\prod_{\ell=i^*+1}^q(\beta_\ell\wedge1)$ and $R:=r\vee(2(\beta^*)^{t^*}t^*)^B$. 
Define
\[
t_0:=\frac{\log 2}{36R^2}\qquad
\text{and}
\qquad
\Gamma:=\left(R+\frac{1}{\sqrt{2\pi t_0}}\right)^d,
\] 
and take the constant $K>0$ appearing in Lemma \ref{lem:minimax}. 
Also, take a $C^\infty$ function $\psi:\mathbb R\to[0,1]$ such that $\psi(x)=1$ for $|x|>2d$ and $\psi(x)=0$ for $|x|\leq d$. 
Increasing the value of $K$ if necessary, we may assume $|\psi(x)-\psi(y)|\leq K|x-y|$ for all $x,y\in\mathbb R^d$. 
Then, given $n\in\mathbb N$, there are an integer $M\geq1$ and functions $f_{(0)},\dots,f_{(M)}\in\ol\mG\big(q, \bd, \bt, \bbeta, K\big)$ satisfying conditions \ref{f-l2}--\ref{f-bound}. 
For each $j=1,\dots,M$, we define the function $b_j:\mathbb R^d\to\mathbb R^d$ as 
\[
b_j(x)=\begin{cases}
f_{(j)}(x)e_{\mathsf i}-r\frac{x}{|x|}\psi(|x|) & \text{if }x\neq0,\\
f_{(j)}(x)e_{\mathsf i} & \text{if }x=0.
\end{cases}
\]
Here, $e_{\mathsf i}$ is the $d$-dimensional unit vector whose $\mathsf i$-th coordinate is 1.
Using \ref{f-smooth}--\ref{f-bound} and properties of $\psi$, one can easily check that $b_0,b_1,\dots,b_M\in\mathcal{B}$ with increasing the value of $K$ if necessary. 
In the following we show that these $b_j$ satisfy \eqref{tsy-eq1}--\eqref{tsy-eq3}. 

First, \eqref{tsy-eq1} immediately follows from \eqref{f-lower}. 
Next, from Section 7.6.4 of \cite{LiSh01}, we have \eqref{tsy-eq2} and
\[
\frac{dP_{j}}{dP_0}(\mathbb{X}_n)=\exp\left\{\int_0^{n\De}\{b_j(X_t)-b_0(X_t)\}^\top dX_t
-\frac{1}{2}\int_0^{n\De}(|b_j(X_t)|^2-|b_0(X_t)|^2)dt\right\}
\] 
for every $j=1,\dots,d$. 
Hence we obtain
\begin{align}
&E_{b_j,n}\left[\log\frac{dP_{j}}{dP_{0}}(\mathbb{X}_n)\right]
\nonumber\\
&=E_{b_j,n}\left[\int_0^{n\De}\langle b_j(X_t)-b_0(X_t), b_j(X_t)\rangle dt
-\frac{1}{2}\int_0^{n\De}(|b_j(X_t)|^2-|b_0(X_t)|^2)dt\right]
\nonumber\\
&=E_{b_j,n}\left[\frac{1}{2}\int_0^{n\De}|b_j(X_t)-b_0(X_t)|^2dt\right]
\nonumber\\
&=\frac{1}{2}E_{b_j,n}\left[\int_0^{n\De}|f_{(j)}(X_t)-f_{(0)}(X_t)|^2dt\right].
\label{eq:kl}
\end{align}
Now, Theorem 1 in \cite{QiRuZh03} implies that
\[
p_j^t(x,y)\leq\prod_{k=1}^d\left(\frac{1}{\sqrt{2\pi t}}\int_{|x_k-y_k|/\sqrt t}^\infty ze^{-(z-R\sqrt t)^2/2}dz\right)
\]
for all $(t,x,y)\in(0,\infty)\times\mathbb R^d\times\mathbb R^d$, where $p_j^t$ are transition densities associated with \eqref{eq:sde} with $b=b_j$. 
Since
\begin{align*}
&\int_{0}^\infty ze^{-(z-R\sqrt t)^2/2}dz\\
&=R\sqrt t\int_{0}^\infty e^{-(z-R\sqrt t)^2/2}dz
-\int_{0}^\infty \frac{d}{dz}e^{-(z-R\sqrt t)^2/2}dz\\
&\leq R\sqrt {2\pi t}
+e^{-R^2t/2},
\end{align*}
we have for any $x,y\in\mathbb R^d$ and $t\geq t_0$
\begin{align*}
p_j^t(x,y)
&\leq\prod_{k=1}^d\left(R+\frac{1}{\sqrt{2\pi t}}\right)
\leq\Gamma.
\end{align*}
Combining this with \eqref{f-upper}, \ref{f-bound}, $\supp(f_{(j)})\subset[0,1]^d$ and Fubini's theorem, we obtain 
\begin{align*}
&E_{b_j,n}\left[\int_0^{n\De}|f_{(j)}(X_t)-f_{(0)}(X_t)|^2dt\right]\\
&=E_{b_j,n}\left[\int_0^{t_0}|f_{(j)}(X_t)-f_{(0)}(X_t)|^2dt\right]
+\int_{\mathbb R^d\times\mathbb R^d}\left\{\int_{t_0}^{n\De}|f_{(j)}(y)-f_{(0)}(y)|^2p^t_j(x,y)dt\right\}\mathcal L_\eta(dx)dy\\
&\leq4R^2t_0+\frac{\log M}{9}.
\end{align*}
Note that \eqref{f-upper} and \eqref{f-lower} imply $M\geq2$. Thus, we have by the definition of $t_0$
\begin{align*}
E_{b_j,n}\left[\int_0^{n\De}|f_{(j)}(X_t)-f_{(0)}(X_t)|^2dt\right]
\leq\frac{2\log M}{9}.
\end{align*}
Combining this with \eqref{eq:kl}, we deduce \eqref{tsy-eq3}. This completes the proof. 
\end{proof}


\appendix
\section{Appendix}

\subsection{Proof of Remark \ref{rem:mixing}(a)}\label{proof:rem:mixing}

By \cite[Proposition 1]{Da73}, we have for any $t\geq0$
\begin{align*}
\beta_X(t)&=\sup_{s\geq0}\int_{\mathbb R^d}\|P^t(x,\cdot)-\mathcal L_\eta P^{s+t}\|\mathcal L_\eta P^s(dx).
\end{align*}
Hence
\begin{align*}
\beta_X(t)\leq\sup_{s\geq0}\int_{\mathbb R^d}\|P^t(x,\cdot)-\Pi\|\mathcal L_\eta P^s(dx)
+\sup_{s\geq0}\|\Pi-\mathcal L_\eta P^{s+t}\|.
\end{align*}
For any Borel set $A\subset\mathbb R^d$,
\[
P^t(x,A)-\Pi(A)=\int_{\mathbb R^d}\{P^{t/2}(y,A)-\Pi(A)\}P^{t/2}(x,dy).
\]
Hence, for any $s\geq0$,
\begin{align*}
&\int_{\mathbb R^d}\|P^t(x,\cdot)-\Pi\|\mathcal L_\eta P^s(dx)\\
&\leq\int_{\mathbb R^d}\left\{\int_{\mathbb R^d}\|P^{t/2}(y,\cdot)-\Pi\|P^{t/2}(x,dy)\right\}\mathcal L_\eta P^s(dx)\\
&=\int_{\mathbb R^d}\left\{\int_{\mathbb R^d}\|P^{t/2}(y,\cdot)-\Pi\|P^{t/2+s}(x,dy)\right\}\mathcal L_\eta(dx)\\
&\leq2\int_{\mathbb R^d}\|P^{t/2+s}(x,\cdot)-\Pi\|\mathcal L_\eta(dx)
+\int_{\mathbb R^d}\|P^{t/2}(y,\cdot)-\Pi\|\Pi(dy)\\
&\leq2e^{-\rho (t/2+s)}\int_{\mathbb R^d}V(x)\mathcal L_\eta(dx)
+e^{-\rho t/2}\int_{\mathbb R^d}V(x)\Pi(dy).
\end{align*}
Also, for any Borel set $A\subset\mathbb R^d$,
\[
\Pi(A)-\mathcal L_\eta P^{s+t}(A)=\int_{\mathbb R^d}\{\Pi(A)-P^{s+t}(x,A)\}\mathcal L_\eta(dx).
\]
Hence
\[
\|\Pi-\mathcal L_\eta P^{s+t}\|\leq\int_{\mathbb R^d}\|\Pi-P^{s+t}(x,\cdot)\|\mathcal L_\eta(dx)
\leq e^{-\rho (s+t)}\int_{\mathbb R^d}V(x)\mathcal L_\eta(dx).
\]
Consequently, 
\[
\beta_X(t)\leq\left\{3\int_{\mathbb R^d}V(x)\mathcal L_\eta(dx)+\int_{\mathbb R^d}V(x)\Pi(dy)\right\}e^{-\rho t/2}.
\]
So Assumption \ref{ass.mixing} holds with $C_\beta'=3\int_{\mathbb R^d}V(x)\mathcal L_\eta(dx)+\int_{\mathbb R^d}V(x)\Pi(dy)$ and $C_\beta=\rho/2$.
\qed

\subsection{Proof of Lemma \ref{lem:approx}}\label{proof:lem:approx}

We construct $\widetilde f_n^*$ in essentially the same way as in the proof of \cite[Theorem 1]{ref.SH}. 


\noindent\textbf{Step 1}. Let $f$ be of the form \eqref{eq.mult_composite_regression}. Define functions $h_i:[0,1]^{d_i}\to[0,1]^{d_{i+1}}$ $(i=0,\dots,q)$ by
\begin{align*}
  h_0 (x) &:= \frac{g_0 (x)}{2K} + \frac 12, \\
  h_i (x) &:= \frac{g_i(2K x - K\boldsymbol1)}{2K} + \frac 12,\qquad i=1,\dots,q-1, \\
  h_q (x) &:= g_q(2K x -K\boldsymbol1),
\end{align*}
where $\boldsymbol1=(1,\dots,1)^\top\in\mathbb R^d$. One can easily check that
\begin{align}
  f = g_q \circ \cdots \circ g_0 = h_q \circ \cdots \circ h_0
  \label{eq.f=h_comp}
\end{align}
and $h_{ij}\in \mC_{t_i}^{\beta_i}([0,1]^{d_i},Q_i)$ for all $i=0,1,\dots,q$ and $j=1,\dots,d_{i+1}$, where $h_i=(h_{ij})_j$ and 
\[
Q_0=1,\quad
Q_i=(2F)^{\beta_i}~(i=1,\dots,q-1),\quad
Q_q=(2F)^{\beta_q+1}.
\]






\noindent\textbf{Step 2}. For $i=0,1,\dots,q$, define
\begin{align*}
  r_i := d_{i+1}(t_i+\lceil \beta_i\rceil ),\qquad
  m_i := \left\lceil \frac{\beta_i + t_i}{t_i} \log_2 (n\De\phi_n) \right\rceil - 6,
\end{align*}
and set $L'_i:= 8+(m_i+5)(1+\lceil \log_2 (t_i\vee \beta_i) \rceil)$. 
Also, take a sufficiently small $c>0$ so that $c_{\bp} \geq 2c\max_{i=0, \dots, q} 6r_i$, and define $N:=\lceil cn\De\phi_n\rceil$. 
There is a constant $M>0$ depending only on $(c_{\bp}, q, \bd, \bt, \bbeta, F)$ such that
\[
  N \geq \max_{i=0,\dots,q} \left( (\beta_i+1)^{t_i} \vee (Q_i+1)e^{t_i} \right)
  \]
  and
  \begin{align}
  c_{\bp} n\De \phi_n \geq d\vee N\max_{i=0, \dots, q} 6r_i,\label{cond:width}
  \end{align}
  when $n\De\geq M$. Increasing $M$ if necessary, we may also assume that $m_i>0$ for all $i$ when $n\De\geq M$. In the following steps, we always assume $n\De\geq M$ unless otherwise stated. 

%

\noindent\textbf{Step 3}. Applying Theorem 5 in \cite{ref.SH} to each function $h_{ij}$, we can construct a network $h_{ij}^* \in \mF(L'_i, (d_i, 6(t_i+\lceil \beta_i\rceil )N,\ldots,6(t_i+\lceil \beta_i\rceil ) N,1), s_i)$ with 
\begin{equation}\label{cond:s}
s_i \leq 141 (t_i+\beta_i+1)^{3+t_i} N (m_i+6),
\end{equation}
such that
\begin{align}
  \| h_{ij}^* - h_{ij}\|_{\infty}\leq (2Q_i+1)(1+t_i^2 +\beta_i^2)6^{t_i} N 2^{-m_i}+ Q_i 3^{\beta_i} N^{-\frac{\beta_i}{t_i}}.
  \label{eq.proof_mt1}
\end{align}
If $i<q$, we define $\widetilde h_{ij} := 1-\sigma(1-\sigma(h_{ij}^*))$. By construction we have $\widetilde h_{ij} \in \mF(L'_i+2, (t_i, 6(t_i+\lceil \beta_i\rceil )N,\ldots, 6(t_i+\lceil \beta_i\rceil )N,1), s_i+4)$ and $\sigma(\widetilde h_{ij}) = (h_{ij}^*\vee 0) \wedge 1.$ Moreover, since $0\leq h_{ij}\leq1$ by construction, we also have
\begin{align}
  \| \sigma(\widetilde h_{ij}) - h_{ij}\|_{\infty}\leq \| h_{ij}^* - h_{ij}\|_{\infty}.
  \label{eq.proof_mt2}
\end{align}


\noindent\textbf{Step 4}. Let $\widetilde h_i =(\widetilde h_{ij})_{j = 1}^{d_{i+1}}$ for $i=0,\dots,q-1$. Computing the networks $\widetilde h_{ij}$ in parallel, we obtain
\begin{align*}
   \widetilde h_i\in\mF \big(L'_i+2, (d_i, 6r_i N,\ldots,6r_i N,d_{i+1}), d_{i+1}(s_i+4) \big).
\end{align*}
We then define $f_n^* = h_{q1}^* \circ \sigma(\widetilde h_{q-1}) \circ \cdots \circ \sigma(\widetilde h_{0})$. We have by construction (cf.~Section 7.1 in \cite{ref.SH})
\begin{align*}
  f_n^* \in\mF \Big( E, (d,  6r_0 N,\ldots,6 r_q N,1), \sum_{i=0}^{q} d_{i+1}(s_i+4) -4 \Big),
\end{align*}
where $E = \sum_{i=0}^q (L'_i + 2 + 1) - 1 - 2 = 3q+\sum_{i=0}^q L'_i.$ Note that
\begin{align*}
  E
  & \leq 3q + 8(q+1) + \sum_{i=0}^q (m_i+5)(\log_2 (t_i \vee \beta_i) + 2) \\
  & = 11q + 8 + \sum_{i=0}^q (m_i+5) \log_2 (4t_i \vee 4\beta_i) 
   \leq L,
\end{align*}
where the last inequality follows from condition (ii) in Theorem \ref{thm.main}. 
Moreover, in the light of condition (ii) in Theorem \ref{thm.main} and \eqref{cond:s}, we can increase the value of $M$ if necessary and may assume for $n\De\geq M$
\[
 \sum_{i=0}^{q} d_{i+1}(s_i+4) -4+(L-E)d
  \leq c_s^l n\De \phi_n \log (n\De).
\]
Note that $M$ now depends on $(c_L^u, c_{\bp}, c_s^l, c_s^u, q, \bd, \bt, \bbeta, F)$. 
Hence, we obtain $f_n^*\in\mF(L, \bp, s)$ by Eq.(18) in \cite{ref.SH} and conditions (iii)--(iv) in Theorem \ref{thm.main}.

\noindent\textbf{Step 5}. Applying Lemma 3 in \cite{ref.SH} with \eqref{eq.proof_mt1} and \eqref{eq.proof_mt2}, we deduce 
\begin{align*}
  \norm{ f_n^* - f}_\infty^2 
  & \leq C_1\bra{ K \prod_{\ell=0}^{q-1} (2K)^{\beta_{\ell+1}}
  \sum_{i=0}^q \big( N^{-\frac{\beta_i}{t_i}} \big)
  ^{\prod_{\ell = i+1}^q \beta_\ell \wedge 1} }^2 \\
  & \leq C_2\max_{i=0,\dots,q} N^{-\frac{2\beta_i^*}{t_i}}
   \leq C_3\phi_n,
\end{align*}
where $C_1,C_2,C_3>0$ depend only on $(c_{\bp}, q, \bd, \bt, \bbeta,F)$. 

Define $\widetilde f_n^* := f_n^* (\frac{\|f\|_\infty}{\| f_n^* \|_\infty} \wedge 1).$ Since $f \in \mG(q, \bd, \bt, \bbeta, K)$, $\|\widetilde f_n^*\|_\infty \leq \|f\|_\infty \leq \|g_q\|_\infty \leq K \leq F$. 
Moreover, $\widetilde f_n^* \in \mF(L, \bp, s, F).$ Writing $\widetilde f_n^* - f = (\widetilde f_n^*- f_n^*) + ( f_n^* -f),$ we obtain $\|\widetilde f_n^* - f\|_\infty^2 \leq 4 \|f_n^* - f\|_\infty^2\leq4C_3\phi_n$ for $n\De\geq M$. Meanwhile, since $\|\widetilde f - f\|_\infty\leq2F$ for any $\widetilde f\in \mF(L, \bp, s, F)$, there is a constant $C_4>0$ depending only on $(c_L^u, c_{\bp}, c_s^l, c_s^u, q, \bt, \bbeta,F)$ such that $\|\widetilde f - f\|_\infty\leq C_4\phi_n$ for all $\widetilde f\in \mF(L, \bp, s, F)$ and $n\De< M$. This completes the proof.
\qed

%


\subsection{Generalization error bounds in Besov spaces}\label{sec:besov}

For a function $f:[0,1]^d\to\mathbb R$, an integer $r\geq0$ and a vector $h\in\mathbb R^d$, we define the $r$-th difference of $f$ in the direction $h$ by
\begin{align*}
\Delta^r_h(f)(x):=
\begin{cases}
\sum_{j=0}^r\binom{r}{j}(-1)^{r-j}f(x+jh) & \text{if }x,x+rh\in[0,1]^d,\\
0 & \text{otherwise}.
\end{cases}
\end{align*}
Note that, if $x,x+rh\in[0,1]^d$, then $x+jh\in[0,1]^d$ for all $j=0,1,\dots,r$ by convexity. 
Then, for every $p\in(0,\infty]$, we define the $r$-th modulus of smoothness of $f$ in $L^p([0,1]^d)$ as
\[
w_{r,p}(f,t):=\sup\{\|\Delta^r_h(f)\|_p:h\in\mathbb R^d,|h|\leq t\},\qquad t>0.
\]
Here, for a function $g:[0,1]^d\to\mathbb R$, we write
\[
\|g\|_p=\begin{cases}
(\int_{[0,1]^d}|g(x)|^pdx)^{1/p} & \text{if }0<p<\infty,\\
\sup_{x\in[0,1]^d}|g(x)| & \text{if }p=\infty.
\end{cases}
\]
Further, for $\alpha>0$, we set
\[
|f|_{B_{p,q}^\alpha}:=\left(\int_0^\infty[t^{-\alpha}w_{r,p}(f,t)]^q\frac{dt}{t}\right)^{1/q}
\]
for $q\in(0,\infty)$ and
\[
|f|_{B_{p,\infty}^\alpha}:=\sup_{t>0}t^{-\alpha}w_{r,p}(f,t).
\]
For $p,q\in(0,\infty]$ and $\alpha>0$, we define the \textit{Besov space} $B_{p,q}^\alpha([0,1]^d)$ as the set of all functions $f:[0,1]^d\to\mathbb R$ such that
\[
\|f\|_{B_{p,q}^\alpha}:=\|f\|_p+|f|_{B_{p,q}^\alpha}<\infty.
\]

\begin{thm}
  \label{thm:besov}
  Suppose that Assumptions \ref{ass.SDE1}--\ref{ass.SDE6} are satisfied. 
  Consider the $d$-dimensional diffusion process model \eqref{eq.mod} such that $\|f_0\|_{B_{p,q}^\alpha}\leq1$ and $\|f_0\|_\infty\leq K$ for some $p,q\in(0,\infty],\alpha>0$ and $K>0$ with $\alpha>d/p$. 
  Set $N:=(n\De)^{d/(2\alpha+d)}$. 
  Take a positive integer $m$ such that $0<\alpha<\min\{m,m-1+1/p\}$, and set $W_0:=6dm(m+2)+2d$. 
  Then, there exists a constant $c_L^l>0$ depending only on $(d, p,q,\alpha,m)$ such that, if $\hf$ is an estimator taking values in the network class $\mF(L, \bp, s, F)$ satisfying
  \begin{enumerate}[label=(\roman*)]
    \item $F \geq K\vee1,$
    \item $c_L^l\log_2^2 N \leq L\leq c_L^u N $ for some $c_L^u>0$,
    \item $NW_0 \leq \min_{i=1, \ldots, L} p_i,$
    \item $s=2c_s\{((L-1)W_0^2+1)N+12L\}$ for some $c_s\geq1$,
  \end{enumerate}
  then
  \begin{align*}
    \eR{\hf}
    \leq 2\Psin{\hf} + C\bra{(n\De)^{-\frac{2\alpha}{2\alpha+d}} L^2 \log^2 (n\De) + \De},
  \end{align*}
  where $C>0$ is a constant depending only on 
  \[
  (C_{b}, C_{b}', C_\Sigma, C_\Sigma', C_\be, C'_\be,  c_L^u,c_s, d, p,q,\alpha,m, F).
  \]
  In particular, if $\hf$ be a minimizer of $\eQ{f}$ over $f\in\mF(L, \bp, s, F)$ and $L=O(\log_2^2(n\De))$, $n\De\to\infty$ and $n\De^2\to0$ as $n\to\infty$, then
\[
    \eR{\hf} = O\bra{(n\De)^{-\frac{2\alpha}{2\alpha+d}}  \log^6 (n\De)}\qquad\text{as }n\to\infty.
\]
\end{thm}

\begin{rem}
The assumption $\alpha>d/p$ is stronger than the corresponding one in Theorem 2 in \cite{ref.Suzuki} ($\alpha>d(1/p-1/2)_+$ is assumed ibidem). 
Note that, however, we typically need to assume a stronger condition $\alpha\geq d/p+1$ to ensure the Lipschitz continuity of the drift function (cf.~Proposition 1 in \cite{SuNi19}), so this difference does not matter in practice. 
\end{rem}

\begin{proof}[Proof of Theorem \ref{thm:besov}]
By Proposition 1 in \cite{ref.Suzuki} (with $s=\alpha$ and $r=\infty$) and Remark \ref{rem:dnn-b}, there exist positive constants $c_L^l$ and $C_0$ depending only on $(d, p,q,\alpha,m)$ such that
\[
\inf_{f\in \mF(L, \bp, s)} \norm{f - f_0}_\infty\leq C_0N^{-\alpha/d}=C_0(n\De)^{-\alpha/(2\alpha+d)}
\]
when the network parameters $L,\bp$ and $s$ satisfy conditions (ii)--(iv). Then, by the same argument as in Step 5 of Section \ref{proof:lem:approx}, we obtain
\[
\inf_{f\in \mF(L, \bp, s,F)} \norm{f - f_0}_\infty\leq 2C_0(n\De)^{-\alpha/(2\alpha+d)}.
\]
Combining this with Theorem \ref{thm.oracle_ineq} gives the first claim. 
The second claim is an immediate consequence of the first one. 
\end{proof}

\paragraph{Acknowledgments.}

The authors are grateful to an anonymous referee for his/her constructive comments. 
The authors also thank Kengo Kamatani for his helpful comments. 
Yuta Koike's research was partly supported by JST CREST Grant Number JPMJCR2115 and JSPS KAKENHI Grant Numbers JP17H01100, JP19K13668, JP22H00834, JP22H01139. 



\end{document}